\documentclass[12pt,leqno]{amsart}
\usepackage{amssymb,amsmath,amsthm,bm}
\usepackage{graphicx}
\usepackage{color}
\usepackage[textsize=tiny]{todonotes}
\usepackage[normalem]{ulem}
\usepackage[bookmarksopen,bookmarksdepth=3,colorlinks,citecolor=red,pagebackref,hypertexnames=true]{hyperref}
\usepackage[msc-links,nobysame,non-sorted-cites, initials]{amsrefs}
\usepackage[inline]{enumitem}
\usepackage{mathtools }
\mathtoolsset{showonlyrefs}

\definecolor{darkblue}{RGB}{0,0,160}

\usepackage[lining]{libertine}   
\usepackage{cabin}
\usepackage[libertine]{newtxmath}
 
\usepackage[T1]{fontenc}

\def\e{{\rm e}}
\def\eps{\varepsilon}

\def\d{{\rm d}}
\def\dist{{\rm dist}}

\def\R {\mathbb{R}}
\def\SS {\mathbb{S}}

\def\M {{\mathrm M}}

\def\1 {{\mbox{\boldmath 1}}}

\def \l {\langle}

\def \r {\rangle}

\def \and{\quad\text{and}\quad}

\def\ind{\cic{1}}
\newcommand{\cic}{\bm}
\def \no#1#2#3 {{\bf #1} (#3), #2.}
\def \eds#1#2#3 {#1, #2, #3.}
\newcounter{counter}

\numberwithin{equation}{section}
\numberwithin{counter2}{section}
\newtheorem{proposition}[subsection]{Proposition}
\newtheorem*{mnest}{The  (TCI)$_{(m,n)}$ estimate}
\newtheorem*{mnpest}{The  (PART)$_{(m,n)}$  estimate}
\newtheorem{theorem}[counter]{Theorem}

\newtheorem{corollary}[subsection]{Corollary}
\newtheorem{lemma}[subsection]{Lemma}

\theoremstyle{definition}

\newtheorem*{remark*}{Remark}
\newtheorem*{warn*}{A word of warning}

\newtheorem{remark}[subsection]{Remark} 
\theoremstyle{plain}
	 \def\Xint#1{\mathchoice
	   {\XXint\displaystyle\textstyle{#1}}%
	   {\XXint\textstyle\scriptstyle{#1}}%
	   {\XXint\scriptstyle\scriptscriptstyle{#1}}%
	   {\XXint\scriptscriptstyle\scriptscriptstyle{#1}}%
	   \!\int}
	 \def\XXint#1#2#3{{\setbox0=\hbox{$#1{#2#3}{\int}$}
	     \vcenter{\hbox{$#2#3$}}\kern-.5\wd0}}
	 
	 \def\avgint{\Xint-}

\numberwithin{figure}{section}

\makeatletter
\let\save@mathaccent\mathaccent
\newcommand*\if@single[3]{%
  \setbox0\hbox{${\mathaccent"0362{#1}}^H$}%
  \setbox2\hbox{${\mathaccent"0362{\kern0pt#1}}^H$}%
  \ifdim\ht0=\ht2 #3\else #2\fi
  }
\newcommand*\rel@kern[1]{\kern#1\dimexpr\macc@kerna}
\newcommand*\widebar[1]{\@ifnextchar^{{\wide@bar{#1}{0}}}{\wide@bar{#1}{1}}}
\newcommand*\wide@bar[2]{\if@single{#1}{\wide@bar@{#1}{#2}{1}}{\wide@bar@{#1}{#2}{2}}}
\newcommand*\wide@bar@[3]{%
  \begingroup
  \def\mathaccent##1##2{%
    \let\mathaccent\save@mathaccent
    \if#32 \let\macc@nucleus\first@char \fi
    \setbox\z@\hbox{$\macc@style{\macc@nucleus}_{}$}%
    \setbox\tw@\hbox{$\macc@style{\macc@nucleus}{}_{}$}%
    \dimen@\wd\tw@
    \advance\dimen@-\wd\z@
    \divide\dimen@ 3
    \@tempdima\wd\tw@
    \advance\@tempdima-\scriptspace
    \divide\@tempdima 10
    \advance\dimen@-\@tempdima
    \ifdim\dimen@>\z@ \dimen@0pt\fi
    \rel@kern{0.6}\kern-\dimen@
    \if#31
      \overline{\rel@kern{-0.6}\kern\dimen@\macc@nucleus\rel@kern{0.4}\kern\dimen@}%
      \advance\dimen@0.4\dimexpr\macc@kerna
      \let\final@kern#2%
      \ifdim\dimen@<\z@ \let\final@kern1\fi
      \if\final@kern1 \kern-\dimen@\fi
    \else
      \overline{\rel@kern{-0.6}\kern\dimen@#1}%
    \fi
  }%
  \macc@depth\@ne
  \let\math@bgroup\@empty \let\math@egroup\macc@set@skewchar
  \mathsurround\z@ \frozen@everymath{\mathgroup\macc@group\relax}%
  \macc@set@skewchar\relax
  \let\mathaccentV\macc@nested@a
  \if#31
    \macc@nested@a\relax111{#1}%
  \else
    \def\gobble@till@marker##1\endmarker{}%
    \futurelet\first@char\gobble@till@marker#1\endmarker
    \ifcat\noexpand\first@char A\else
      \def\first@char{}%
    \fi
    \macc@nested@a\relax111{\first@char}%
  \fi
  \endgroup
}
\makeatother

\begin{document}

\title[Directional operators on varieties]{Maximal directional operators \\along algebraic varieties}

\author[F. Di Plinio]{Francesco Di Plinio} \address{\noindent Department of Mathematics, University of Virginia, Box 400137, Charlottesville, VA 22904, USA}
\email{\href{mailto:francesco.diplinio@virginia.edu}{\textnormal{francesco.diplinio@virginia.edu}}}
\thanks{F. Di Plinio is partially supported by the National Science Foundation under the grants   NSF-DMS-1650810,  DMS-1800628, and DMS-2000510}

\author[I. Parissis]{Ioannis Parissis}
\address{Departamento de Matem\'aticas, Universidad del Pa\'is Vasco, Aptdo. 644, 48080 Bilbao, Spain and Ikerbasque, Basque Foundation for Science, Bilbao, Spain}

\email{\href{mailto:ioannis.parissis@ehu.es}{\textnormal{ioannis.parissis@ehu.es}}}
\thanks{I.\ Parissis is partially supported by the project PGC2018-094528-B-I00 (AEI/FEDER, UE) with acronym ``IHAIP'', grant T1247-19 of the Basque Government and IKERBASQUE.}


%%%%%%%%%%%%%%%%%%%%%%%%%%%%%% ABSTRACT ABSTRACT ABSTRACT
\begin{abstract}
We establish the sharp growth order, up to epsilon losses, of the $L^2$-norm of the maximal directional averaging operator along a finite subset $V$ of a polynomial variety of arbitrary dimension $m$, in terms of cardinality.
This is an extension of the works by C\'ordoba, for one-dimensional manifolds, Katz for the circle in two dimensions, and Demeter for the 2-sphere. For the case of directions on the two-dimensional sphere we improve by a factor of $\sqrt{\log N}$ on the best known bound, due to Demeter, and we obtain a sharp estimate for our model operator. Our results imply new $L^2$-estimates for  Kakeya type maximal functions with tubes pointing along polynomial directions. Our proof technique is novel and in particular incorporates an iterated scheme of polynomial partitioning on varieties adapted to directional operators, in the vein of Guth,    Guth-Katz, and Zahl. 
\end{abstract}
%%%%%%%%%%%%%%%%%%%%%%%%%%%%%% ABSTRACT ABSTRACT ABSTRACT

\maketitle

%%%%%%%%%%%%%%%%%%%%%%%%%%%%%% SECTION SECTION SECTION
\section{Main results, motivation, background and techniques}\label{sec:intro} We are interested in maximal directional averaging operators, defined with respect to a given set of directions. Our focus is on the higher dimensional setting where the directions are distributed on algebraic varieties in $\mathbb R^n$, of any given codimension. 

\subsection{Main results} More precisely, let $n\geq 2$ and define the   directional averages of a smooth function $f$ on $\R^{n}$  by
\[
\langle f\rangle_{v}(x)\coloneqq  \avgint_{-1} ^1    f (x-t v )  \, \d t, \qquad x,v\in \R^{n},  
\]
and consider  the maximal averaging   operator, at unit  scale, associated to a set of directions $V\subset \mathbb \R^{n} $ 
\begin{equation}\label{e:max}
\mathrm{M}_{V} f(x) \coloneqq \sup_{v\in V}  \, \langle |f|\rangle_{v}(x) .
\end{equation}
It is customary to study the \emph{single scale operator} \eqref{e:max} for sets of directions  $V$ which are normalized to live in  a fixed but arbitrary annular region excluding the origin. To state our results, we will use the unit annulus
\[
\mathcal{A}_n(1)\coloneqq \big\{v \in \R^n:\,1\leq |v|<2\big\}.
\]
The first main result of this article is the sharp  bound in terms of the cardinality parameter $N$, up to arbitrarily small losses, for the maximal operator norm
\[
 \sup\left\{\left\|\mathrm{M}_{V} \right\|_{L^2(\R^{n}) }:V\subset Z_m \cap \mathcal{A}_n(1),\, \#V \leq N^m   \right\}
\]
when $Z_m\subseteq \R^{n}$  is a real algebraic variety of a fixed dimension $1\leq m\leq n-1$. We send to Section \ref{sec:alggeom} for the precise definition of a real algebraic variety of dimension $m$; here, we  restrict ourselves to mentioning  a prototypical example. If $D\geq 1$,  we say that $Z_m$ belongs to the class $\mathcal Z_{m,n}^\times(D)$ if 
\[
Z_{m}=\big\{x\in \R^n:\,P_{1}(x)=\cdots=P_{n-m}(x)=0\big\}
\]
where $P_{1},\ldots,P_{n-m}$ are polynomials in $n$ real variables of degree at most $D$, and  the tangent space to $Z_m$ is $m$-dimensional at all points $x\in Z_m$, in the sense of \eqref{e:wedge1} below. For the class   $\mathcal Z^\times_{m,n}(D)$  we are able to obtain a uniform bound on the operator norm. In particular, polynomial graphs over $m$ variables in $\R^n$, of degree at most $D$,
\[
Z_m =\big\{(y,h(y)):\,y \in \R^m \big\}, \quad h:\R^m\to \R^{n-m} \textrm{  polynomial of degree at most } D,
\] 
belong to the class $\mathcal Z^\times_{m,n}(D)$.
%%%%%%%%%%%%%%%%%%%%%%%%%%%%%% THEOREM THEOREM THEOREM
\begin{theorem}  \label{t:main:manifold} 
Let $n\geq 2$ and $ 1\leq m\leq n-1$.  Let  $Z_m\subseteq \R^n$ be a real algebraic variety of dimension $m$. Then for all $\eta>0$ there is a constant  $\Theta=\Theta(Z_m,\eta)$ such that

\begin{equation} \label{e:main:growth}
  \sup_{\substack{V\subset Z_m\cap \mathcal{A}_n(1)\\ \#V\leq N^m }}\left\|\mathrm{M}_{V} \right\|_{L^2(\R^{n})} \leq \Theta N^{\frac{m-1}{2}+\eta}.
\end{equation}
The constant $\Theta=\Theta(Z_m,\eta)$ depends on $\eta$ and on explicit algebraic properties of the variety $Z_m$.

Furthermore, if $D\geq 1$ and $\mathcal Z^\times_{m,n}(D)$ is the class of real algebraic varieties defined above,  we have the uniform bound
\begin{equation} \label{e:main:uniform}
\sup_{Z_m\in\mathcal Z^\times _{m,n}(D)}\,  \sup_{\substack{V\subset Z_m\cap \mathcal{A}_n(1)\\ \#V\leq N^m }}\left\|\mathrm{M}_{V} \right\|_{L^2(\R^{n})} \leq \Theta   N^{\frac{m-1}{2}+\eta}
\end{equation}
for all $\eta>0$; the constant $\Theta=\Theta(m,n,D,\eta)$ depends only on the dimension parameters $m,n$, on the degree $D$, and on $\eta>0$.
\end{theorem}
%%%%%%%%%%%%%%%%%%%%%%%%%%%%%% THEOREM THEOREM THEOREM

As we will see in the subsequent section, the constant $\Theta(Z_m,\eta)$ that appears in the statement of Theorem~\ref{t:main:manifold} depends on certain notions of  \emph{degree}  and   \emph{count} of the variety  $Z_m$. We refer the reader to \S\ref{sec:alggeom} for precise definitions and further discussion.

Let $\mathrm{M}_{V,r}$ be the analogue of $\mathrm{M}_V$ of \eqref{e:max}  obtained from the directional averages at scale $r>0$.   Scaling shows that $ \|\mathrm{M}_{V,r}\|_{L^p(\R^n)} = \|\mathrm{M}_{V}\|_{L^p(\R^n)}$. When $f$ is the indicator of the unit ball in ${\R^{n}}$, there holds
\[
\mathrm{M}_{V,N} f(x) \gtrsim \frac{1}{|x|}, \qquad \frac{N}{2}< |x| <N,
\] 
if  ${V\subset \SS^{n-1}}$  is a $\frac{c}{N}$-net with $c>0$ sufficiently small; we then gather that \eqref{e:main:growth}, \eqref{e:main:uniform} are sharp up to the $\eta$-correction {for the codimension one case, $m=n-1$. For general codimensions we can adjust the example above by taking $V\subset \SS^m\subset \SS^{n-1}$ and $f\coloneqq f_1 \otimes f_2:\R^{m+1}\times \R^{n-(m+1)}\to \R$ where $f_1$ is the indicator of the unit ball in $\R^{m+1}$ and $f_1$ is a smooth bump function in $\R^{n-(m+1)}$ which is identically $1$ on the unit ball of $\R^{n-(m+1)}$. This modification also proves the sharpness of  \eqref{e:main:growth}, \eqref{e:main:uniform}, up to the $\eta$-correction for  general codimensions.} In fact, a logarithmic correction is necessary in the case $m=1$; see \cite[Proposition 1.3]{Cor77}. 

All cases of Theorem \ref{t:main:manifold} are new, except for the case $m=1$,  which dates back to the work of C\'ordoba \cite{Cor1982} and Barrionuevo \cite{BarrPAMS}, and the case of $Z=\mathbb S^2 $, which is the main result of \cite{Dem12} by Demeter. In the latter case, we prove a more precise result improving on the logarithmic correction of \cite{Dem12}. To describe this improvement,  we introduce the following iterated logarithmic function; for integers $k,N\geq 1$ we set
\[
\log^{[1]}N\coloneqq \log(2+N),\qquad \log^{[k]}N\coloneqq \log(2+\log^{[k-1]}N).
\]
%%%%%%%%%%%%%%%%%%%%%%%%%%%%%% THEOREM THEOREM THEOREM
\begin{theorem}
\label{t:main:card}  For every $k\geq 1$ there exists a constant $\Theta_k>1$  such that
\begin{equation} 
\label{e:main:card}
\sup_{\substack{V\subset \mathbb S^2 \\ \#V \leq N^2}} \left\|\mathrm{M}_{V } \right\|_{L^2(\R^{3}) } \leq \Theta_k N^{\frac12} \sqrt{\log N}  \log^{[k]} N .
\end{equation}
\end{theorem}
%%%%%%%%%%%%%%%%%%%%%%%%%%%%%% THEOREM THEOREM THEOREM

%%%%%%%%%%%%%%%%%%%%%%%%%%%%%% REMARK REMARK REMARK
As customary, we prove our estimates for a Fourier analogue of $\mathrm{M}_{V }$, namely the maximal directional multiplier operator $\mathrm{A}_{V }$ defined in  \eqref{e:fourierav} below, which dominates $\mathrm{M}_{V }$ pointwise on the cone of positive functions.  If $f$ has frequency support in the annulus $\{\xi \in \R^n:\,s<|\xi|<2s\}$, our proof yields the stronger estimate
\[
\sup_{\substack{V\subset \mathbb S^2 \\ \#V \leq N^2}} \left\|\mathrm{A}_{V } f \right\|_{L^2(\R^{3}) } \lesssim_k N^{\frac12}   \log^{[k]} N \, \|f\|_{L^2(\R^3)}
\] 
uniformly over $s>0$. Modifying the example discussed after the statement of Theorem~\ref{t:main:manifold}, namely testing the operator norm on a suitable  frequency cutoff and modulation of the indicator of the unit ball in $\R^{3}$, reveals that the $N$-dependence of the  latter estimate is sharp modulo the iterated logarithmic correction; see \cite{LacLi:tams} for details.

Theorem \ref{t:main:manifold} yields $L^2(\R^n)$-bounds for the Nikodym maximal function
\[
\mathrm{M}_{ Z_m,\delta} f(x) = \sup_{x\in T\in \mathcal T( Z_m,\delta)} \frac{1}{|T|} \int_{T} |f|
\]
where $ \mathcal T( Z_m,\delta)$ is the collection of tubes $T$ with length $1$ and $n-1$-dimensional cross section of width $0<\delta\ll 1$, where the long side oriented along some direction $v\in {Z_m}\cap \mathcal A_{n}(1) $. 

%%%%%%%%%%%%%%%%%%%%%%%%%%%%%% THEOREM THEOREM THEOREM
\begin{theorem} \label{t:nykodym} Let  $Z_m\subseteq \R^n$ be a real algebraic variety of dimension $1\leq m\leq n-1$. Then for all $\eta>0$ there is a constant  $\Theta=\Theta(Z_m,\eta)$ independent of $\,0<\delta\ll 1$, such that
\[
 \left\| {\mathrm{M}_{ Z_m,\delta} } \right\|_{L^2(\R^{n}) } \leq \Theta \delta^{-\frac{(m-1)}{2}{-}\eta}.
\]
\end{theorem}
%%%%%%%%%%%%%%%%%%%%%%%%%%%%%% THEOREM THEOREM THEOREM

%%%%%%%%%%%%%%%%%%%%%%%%%%%%%% PROOF PROOF PROOF
\begin{proof} Let $\delta>0$ and $V\subset Z_m$ be a $\delta$-net with $\# V \lesssim \delta^{-m}$. A well-known reduction \cite{Cor1982,Dem12} yields that
 \[ \left\| {\mathrm{M}_{ Z_m,\delta} } \right\|_{L^2(\R^{n}) } \lesssim  \left\| {\mathrm{M}_{ V } } \right\|_{L^2(\R^{n}) } \]
so that the claim follows by an application of Theorem \ref{t:main:manifold}.
\end{proof}
%%%%%%%%%%%%%%%%%%%%%%%%%%%%%% PROOF PROOF PROOF
The exponent in Theorem \ref{t:nykodym} is in general optimal up to the arbitrarily small $\eta$-loss. {Indeed, notice that the example provided after the statement of Theorem~\ref{t:main:manifold} with $N\sim 1/\delta$ also applies to the bounds of Theorem~\ref{t:nykodym} above. This is because the Nikodym maximal operator is defined with respect to $\delta$-tubes; one then can reduce the supremum in the definition of the Nikodym maximal function to be taken over a subset of $\delta$-tubes pointing along a $\delta$-net on the unit sphere; see \cite[p. 718]{Dem12} for the details of this argument.}

 This result appears to be new in all cases except for $m=n-1$, which is a classical estimate of C\'ordoba \cite{Cor77}, and $m=1$ \cite[Theorem B]{Cor1982}, also due to C\'ordoba.

The converse direction of the implication leading from  Theorem~\ref{t:main:manifold} to  Theorem~\ref{t:nykodym}, namely obtaining sharp bounds for the \emph{thin} averages of \eqref{e:max} from sharp bounds for \emph{thick} averages over tubes, as the ones involved in the definition of the Nikodym maximal operator, is not feasible as the averages \eqref{e:max} are more singular; see also \cite{Dem12}. An exception to this heuristic is, in two dimensions, the case of uniformly distributed directions on the circle $\SS^1$. In this setting, the difference between thin and thick averages is a square function which, at least in $L^2(\R^2)$, can be easily treated by overlap considerations; see also the discussion below in \S\ref{sec:backg}.

%%%%%%%%%%%%%%%%%%%%%%%%%%%%%% SECTION SECTION SECTION
\subsection{Motivation} In addition to the intrinsic relevance of Theorems \ref{t:main:manifold}, \ref{t:main:card} and \ref{t:nykodym} to  the realm of classical differentiation theory, this paper finds  motivation  within  a more general program, aimed at  understanding  directional maximal and singular integral operators in higher dimensions, \emph{under no particular structural assumptions}. This is in contrast with  previous results in the literature, for instance   \cites{BarrPAMS,Barr93,Cor1982}, which are obtained   for sets of directions complying with some type of geometric configuration. In the papers cited above, the structure is that of uniform distribution of the discrete set  on the sphere. On the other hand, in e.g. \cite{PR2} and references therein, the set of directions is allowed to be infinite but is assumed to be lacunary. In both cases, these   conditions  allow for an effective splitting into subsets which are either of controlled cardinality, as in the case of uniform distribution, or exhibit self-similar behavior as in the case of lacunary directions. Both families of results mentioned above, as well as their two-dimensional counterparts as in \cite{Alf} and \cites{ASV1,ASV}, show that a very effective way to handle directional operators is the divide and conquer technique. 

For arbitrary sets of directions in higher dimensions this technique presents a new challenge: an efficient partitioning of a non-uniformly distributed, not naturally ordered set of directions on $\mathbb S^2$, for instance,   is far from obvious. Inspired by  recent developments   in harmonic analysis, most notably the improvements on the restriction problem by Guth \cites{G1,Guth2},  our proofs are based on a polynomial partitioning scheme adapted to maximal directional operators. We establish a novel strategy for handling these operators in a very general setting: the directions lie on an algebraic variety of arbitrary codimension in $\R^n$, and essentially no other structure is assumed. To exemplify the intrinsic gain that is brought by this perspective, we mention  Theorem \ref{t:main:card},    which is about directions on the sphere $\mathbb S^2$ and  where no \emph{a-priori} relevant algebraic structure is present. The improvement over \cite{Dem12} is obtained by   partitioning our directions into connected components of the complement of the zero set of a polynomial $P$ on the sphere, each containing a roughly constant number of points, and with favorable overlap properties ---not too many cells intersecting  each given hyperplane--- and using the algebraic structure, see the dimension 1 estimate of Theorem \ref{thm:curve}, to handle the contribution of those directions falling on the partitioning zero set. 

We believe that our methods can be furthered to obtain $L^p$-estimates for more general singular and maximal averages in higher dimensions, when there is no  structure in the set of directions.    The natural point of view within the polynomial method is that of systematically attacking the cases of directions lying on algebraic varieties in $\R^n$ of arbitrary codimension.  As an  example of  relevant open question,  it is worth noting that no nontrivial bound is currently available for  the  \emph{multi-scale directional maximal operator} along directions lying on a polynomial  subvariety of $\SS^{n-1}$,  without any additional structure.

%%%%%%%%%%%%%%%%%%%%%%%%%%%%%% SECTION SECTION SECTION
\subsection{Background}\label{sec:backg} The study of directional maximal  and singular integrals initially arose as a natural companion to questions on Kakeya-type maximal operators, Bochner-Riesz multipliers, and the conjectures of Zygmund and Stein for corresponding objects defined along 2-dimensional vector fields: we provide a short overview with particular focus on the $L^2$-theory. In \cite{Cor77} C\'ordoba showed that the maximal average defined with respect to rectangles in the plane, of fixed eccentricity $\delta$, is bounded on $L^2(\R^2)$ with an operator norm of the order $(\log\delta^{-1})^{\frac12}$; the bound above is best possible as revealed by a counterexample constructed by means of the Kakeya set. We note here that, in two dimensions, this operator is essentially equivalent to a maximal averaging operator along uniformly $\delta$-spaced directions in $\SS^1$, their difference being that of an easy to treat square function; it is also of some importance to highlight that this comparison is not as readily available in higher dimensions, nor in the case of more general algebraic varieties in place of the sphere. Soon after the result of \cite{Cor77}, Str\"omberg, \cite{Stromberg}, showed that the $L^2(\R^2)\to L^{2,\infty}(\R^2)$ norm of the maximal average with respect to rectangles pointing in a uniformly distributed set of $N$ directions in $\SS^1$, without any restriction on their eccentricity, is of the order $(\log N)^\frac12$, and this is best possible. The numerology here is $\delta=N^{-1}$. These results were generalized by Katz in \cites{Katz,KB}, where the author proved the same sharp bound for the $L^2(\R^2)\to L^{2,\infty}(\R^2)$ bound for the operator norm of the maximal average with respect to an \emph{arbitrary} set of $N$ directions in $\SS^1$. We note here that $L^2$ is the critical $L^p$-space for the Kakeya, or Nikodym maximal operators in $\R^2$, while in $\R^n$ the critical exponent is $p=n$; see \cite{Tao}.

Parallel to the results above was the investigation of the maximal averages in the plane given by a set of directions which is \emph{infinite}, but possesses certain arithmetic-geometric structure. Collectively Str\"omberg, \cite{Strlac}, R.\ Fefferman and C\'ordoba, \cite{CorFeflac}, Nagel, Stein, and Wainger, \cite{NSW}, and Sj\"ogren and Sj\"olin, \cite{SS}, proved that that such maximal operators are bounded on $L^p(\R^2)$ whenever the set of directions is a lacunary set of finite order.   The striking result of Bateman, \cite{Bat}, characterized lacunary sets of finite order in the plane as the only sets that give rise to bounded directional maximal operators on $L^p(\R^2)$ for some (equivalently any) $p\in(1,\infty)$.

In the higher (ambient space $\R^n$, $n\geq 2$) dimensional setting, the picture is far  from complete. We note that some special cases of higher-dimensional lacunary sets of directions were introduced by Nagel, Stein, and Wainger in \cite{NSW}, and by Carbery in \cite{Carbery}, where the authors showed the boundedness of the corresponding maximal operator. However, it was only recently that Parcet and Rogers, \cite{PR2}, gave a general definition of lacunary sets of directions in any dimension and proved the boundedness of maximal operators along such directions in $L^p(\R^n)$, for all $p\in (1,\infty)$ and all $n\geq 2$. Recently in \cite{DPP2}, the authors of the present paper obtained the best possible bound for the Hilbert transform along sets of directions in $\SS^2$ that are lacunary in the sense of \cite{PR2}. 

For averaging operators with respect to arbitrary sets of directions in dimensions $n\geq 2$ Demeter has studied the case of maximal  directional averages along sets $V\subset \SS^2$ and showed $L^2$-bounds of the order $N^\frac12(\log N)$ for the case that $\#V=N^2$. This result is improved by a $\sqrt{\log N}$, modulo iterated logarithmic losses, in Theorem~\ref{t:main:card} of the present paper. Also relevant for us is the result of C\'ordoba, \cite{Cor1982}, for equispaced directions lying on a smooth parametrizable curve in $\SS^{n-1}$, proving $L^2$-bounds of the order $(\log N)^2$, for the corresponding multiscale maximal operator. This bound was improved by Barrionuevo in \cite{BarrPAMS} to the best possible order $\log N$. Finally, Barrionuevo in \cite{Barr93}, has proved almost optimal $L^2$-bounds in arbitrary dimension for the case of $N$ uniformly distributed directions in $\SS^{n-1}$.

%%%%%%%%%%%%%%%%%%%%%%%%%%%%%% SECTION SECTION SECTION
\subsection{Techniques} The bulk of the paper is devoted to the proofs of Theorems \ref{t:main:manifold} and \ref{t:main:card}. These proofs have a similar coarse structure. A finite set of directions $V$ lying on an algebraic variety $Z\subset \R^n$ of dimension $m$ (the sphere in Theorem \ref{t:main:card}) is partitioned by the zero set of a polynomial $P$ into cells $\mathsf{C}$, each containing at most a fixed portion of the original points, plus points lying exactly on (or sufficiently close to) the zero set of $P$. These lie on an algebraic variety of dimension $\mu=m-1$ and their contribution is estimated by induction, when $\mu>1$, or by direct methods, when $\mu=1$;  see for example Theorem \ref{thm:curve}. 

 For the cellular part, the estimate is based on the observation that the directions contributing at a given frequency point $\xi$ are localized on a fattening of the hyperplane perpendicular to $\xi$. If, at first approximation, we ignore the complications brought by the fattening, we   can estimate how many of the $V\cap \mathsf{C}$ operators  overlap at $\xi$ by counting, via Milnor-Thom-type theorems, in how many  components the variety $\xi^\perp \cap Z$ is split by the zero set of $P$. When $Z=\mathbb S^2$ for instance,  $\xi^\perp \cap Z$ is a circle, and the number of zeros of $P$ on the circle is controlled by the degree of $P$.
 In the realistic case, we are dealing with fat hyperplanes and the additional term consisting  of those cells staying close to a $(n-1)$-dimensional hyperplane is handled by comparison with directions lying on an algebraic subvariety of the hyperplane: an approximate projection of $Z$. This term is dealt via induction on the ambient space dimension $n$.

Polynomial partition on subvarieties of  $\R^n$ is in general challenging and an optimal procedure is not yet completely understood, especially in codimension 2 and higher: see   the articles \cite{BaSo,FPS+,MP,Zahl2} and references therein for recent developments. Motivated by the study of the Fourier restriction operator in higher dimensions, Guth \cite{Guth2} introduced a polynomial partitioning scheme based on \emph{transverse complete intersections} (TCI), namely algebraic varieties of the classes $\mathcal Z_{m,n}^\times(D)$ defined in Section \ref{sec:alggeom}, which in effect allows one to partition densities supported on $m$-varieties as if they were on $\R^m$, by keeping some uncertainty in the partitioning polynomials. Similar procedures have been used in \cite{OW}. We adapt the scheme of \cite{Guth2} to our context of directional operators, the final result in this sense being Proposition \ref{prop:polpart}.   The first main additional difficulty, compared to \cite{Guth2,OW}, is that we are partitioning points and not densities.  We tackle this issue by replacing points from the initial polynomial wall with  nearby points, sitting on a nicer wall of the TCI-type;  this replacement can be made harmless for our quantitative estimates. The second is that the approximate projection procedure hinted at above does not preserve TCIs. We remedy this by covering an arbitrary algebraic variety of dimension $m$ with an arbitrarily small neighborhood of a  union of a controlled number of TCIs of dimension no more than $m$; see Proposition \ref{p:smoothapprox}.

Interestingly, the recent article by Katz and Rogers \cite{KatzRogers} solves a conjecture of Guth \cite{Guth2} concerning the maximal number of $\delta$-separated $\delta$-tubes contained in the $\delta$-neighborhood of an algebraic subvariety of $\R^n$. While their estimate is somewhat dual to that of Theorem \ref{t:nykodym}, none of the two can be promptly reduced to the other one. However, remarkably, the proof techniques of the main result of \cite{KatzRogers} share some aspects with our arguments of Section \ref{sec:alggeom}: in particular, the authors of \cite{KatzRogers} appeal to Tarski's quantifier elimination principle within the category of semialgebraic sets, combined with Gromov's complexity estimate. In contrast, inspired by arguments in \cite{MP} by Matou{\v s}ek and Patakova, we construct (approximate) projections of algebraic varieties by explicitly computing elimination ideals  of the original variety, via Gr\"obner bases and the related quantifier elimination theorem; see  \cite{CLO2,CLO} and references therein. This technique allows us to work within the class of algebraic sets.

%%%%%%%%%%%%%%%%%%%%%%%%%%%%%% SECTION SECTION SECTION
\subsection{Structure}Section \ref{sec:antools} contains some preliminary analytic tools and a version of polynomial partition on $m$-dimensional polynomial subvarieties $\R^n$ adapted to our problem, and involving directional averages; see Proposition \ref{prop:polpart}. The proof of Theorem~\ref{t:main:card} is given in Section~\ref{sec:3dgeneraldirs}. Section \ref{sec:alggeom} contains several definitions and tools from algebraic geometry. The proof of Proposition \ref{prop:polpart}, which will find use in  the proof of Theorem~\ref{t:main:manifold}, can also be found in Section~\ref{sec:alggeom}. Finally, the proof of Theorem~\ref{t:main:manifold} is given in the final Section~\ref{sec:mainpf}.

\subsection*{Acknowledgments} The authors want to express their gratitude to the reviewers for their careful reading and suggestions which contributed to improvements in the presentation. 

%%%%%%%%%%%%%%%%%%%%%%%%%%%%%% SECTION SECTION SECTION
\section{Preliminaries} \label{sec:antools}

%%%%%%%%%%%%%%%%%%%%%%%%%%%%%% SECTION SECTION SECTION
\subsection{Notation and recurring definitions}\label{sec:not} By
$\Theta_{\alpha_1,\ldots,\alpha_j}$ we denote a positive constant, possibly depending on the parameters $\alpha_1,\ldots,\alpha_j$ only, which may differ  at each occurrence. We also write
\[
\begin{split}
&
A \lesssim_{\alpha_1,\ldots,\alpha_j} B \iff A \leq \Theta_{\alpha_1,\ldots,\alpha_j} B, \qquad
\\ &A \sim_{\alpha_1,\ldots,\alpha_j} B \iff A \lesssim_{\alpha_1,\ldots,\alpha_j} B, \, B \lesssim_{\alpha_1,\ldots,\alpha_j} A .
\end{split}
\]
{For a bounded operator $T:L^2(\R^n)\to L^2(\R^n)$ we use the shorthand notation $\|T\|_{L^2(\R^n)}\coloneqq \|T:L^2(\R^n)\to L^2(\R^n)\|$.} If $v\in V\subset \R^n\setminus \{0\}$ we write $v'\coloneqq v/|v|$ and $V'\coloneqq\{v':\, v\in V\}\subset \mathbb S^{n-1}$.
We use the notation
\[
\mathrm{dist}(U,V) \coloneqq \sup_{u\in U} \inf_{v\in V} |u-v|
\]
for the (asymmetric) distance between $U,V\subset \R^n$.
We will be working with  (frequency and directional)  annuli: for $R\geq 1,$
\[\mathcal A_{n}(R)\coloneqq\big\{\xi \in \R^{n}:\,R^{-1}\leq |\xi|<2R\big\}.\]
We define the  frequency bands  
\begin{equation}\label{e:bandsR}
R_{\xi,s}\coloneqq\left\{\eta \in  \mathcal A_n(1):\,|\xi\cdot \eta|<s|\eta|\right\}, \qquad \xi \in \R^{n}\setminus \{0\},\quad s>0.
\end{equation} 
Let $A_{v,s}$ be the operator defined in \eqref{e:fourierav} below and $S_1$ be a smooth frequency cutoff adapted to the annulus $\mathcal A_n(1)$. Abusing notation we will write $\widehat{S_1}$ for the (smooth) Fourier multiplier of the operator $S_1$; we will be using repeatedly that the Fourier  support of $A_{v,s} \circ S_1$ is contained in the (fat) band $R_{v,s}$.
For $R\subset \mathbb R^{n}$, we write
\begin{equation}
\label{e:notationR}
f_R(x) \coloneqq \int_R \widehat{f}(\xi) \e^{ix\cdot \xi }\, \d \xi. 
\end{equation}
to indicate the rough frequency restriction to $R$.

%%%%%%%%%%%%%%%%%%%%%%%%%%%%%% SECTION SECTION SECTION
\subsection{Single scale averages} Our tools will be largely Fourier analytic in nature: in fact, we will work with  maximal operators obtained by replacing the rough averages in \eqref{e:max} with the smooth directional Fourier multipliers 
\begin{equation} \label{e:fourierav}
\mathrm{A}_{V,s} f \coloneqq \sup_{v\in V} |A_{v,s} f|, \qquad  A_{v,s} f(x) \coloneqq \int_{\R^{n }} \widehat f(\xi) \psi \left(\frac{\xi \cdot v}{s}\right) \e^{ix\cdot\xi}\, \d \xi
\end{equation}
with  $x\in\R^n$ and $s>0$. 
Throughout,   $\psi:\R\to \R$ will be the Fourier transform of an even nonnegative Schwartz function $\Psi:\R\to \R$ with the following properties: 
\[
\mathrm{supp}\, \psi \subset [-2^{-10},2^{-10}], \qquad  \Psi(t)\geq \cic{1}_{(-1,1)}(t).
\]
This smooth function  will remain fixed throughout the paper so we suppress any implicit dependence on the choice of $\psi$.
We are chiefly interested in the case $s=1$, where we suppress the subscript and simply write $A_{v}$, $\mathrm{A}_V$ in place of $A_{v,1}$, $\mathrm{A}_{V,1}$.  By scaling $f$, we see that
\begin{equation}
\label{e:scale}
\|\mathrm{A}_{V,s}\|_{L^p(\R^n)} = \|\mathrm{A}_{V}\|_{L^p(\R^n)} \qquad \forall s>0,\quad 0<p\leq \infty.
\end{equation} 
This scaling invariance is destroyed when restricting the operators $\mathrm{A}_{V,s}$ to act on functions with  frequency support in $\mathcal A_n(1)$, and for this reason we work with a more general scale parameter below; see for instance Proposition \ref{p:CWW}.
 
With our choice of $\psi$, we have the following  comparison principle; the  upper bound is immediate by pointwise comparison for $f\geq 0$, while the lower bound follows from \eqref{e:lemmainter2} below, and scaling.
%%%%%%%%%%%%%%%%%%%%%%%%%%%%%% LEMMA LEMMA LEMMA
\begin{lemma} \label{l:roughsmooth} Let  $V\subset \R^n$.  Then
\[
 \begin{split}
  \left\|\mathrm{M}_{V}  \right\|_{L^2(\R^{n })} \sim_n \left\|\mathrm{A}_{V}  \right\|_{L^2(\R^{n })}.
\end{split}
\]
\end{lemma}
%%%%%%%%%%%%%%%%%%%%%%%%%%%%%% LEMMA LEMMA LEMMA

%%%%%%%%%%%%%%%%%%%%%%%%%%%%%% SECTION SECTION SECTION
\subsection{Decoupling of frequency annuli and localization} Let   ${S_1}$ be a smooth Littlewood-Paley radial cutoff to the annulus $\mathcal A_{n}(1)$. Using the Chang-Wilson-Wolff inequality as in \cite{Dem,DPP2,DPP,GHS}, and scaling, we are able to reduce our maximal estimates to functions with frequency support in $\mathcal A_{n}(1)$, as recorded in the following proposition. 

%%%%%%%%%%%%%%%%%%%%%%%%%%%%%% PROPOSITION PROPOSITION PROPOSITION
\begin{proposition} \label{p:CWW}  Let    $V\subset  \R^n$ be a finite subset. Then   
\[
  \left\|\mathrm{A}_{V}  \right\|_{L^2(\R^{n})} \lesssim_n \sqrt{\log  \#V  } \sup_{s>0} \left\|\mathrm{A}_{V,s}\circ {S_1}  \right\|_{L^2(\R^{n})}.
\]
\end{proposition}
%%%%%%%%%%%%%%%%%%%%%%%%%%%%%% PROPOSITION PROPOSITION PROPOSITION

The main advantage of the reduction of Proposition \ref{p:CWW} is that, in addition to being Fourier localized, the averages  \eqref{e:fourierav} also enjoy a localization property with respect to the direction $v$, in the sense of the next lemma. {Below we denote by $\M$ the usual Hardy-Littlewood maximal function in $\R^n$
\[
\M f(x)\coloneqq \sup_{r>0} \frac{1}{r^n}\int_{|t|\leq r}|f(x-t)|\,\d t,\qquad x\in\R^n. 
\]}

%%%%%%%%%%%%%%%%%%%%%%%%%%%%%% LEMMA LEMMA LEMMA
\begin{lemma}\label{lem:closeness} Let  $s>0$ and $0<a<2^{3}$. Let $U,V\subset \mathcal A_n(2)$ have the property that
\begin{equation} \label{e:red:closeness}
\mathrm{dist}\big(U',V'\big) <as .
\end{equation}
where $U'\coloneqq \{u/|u|:\,u\in U\}$ and similarly for $V$. Then 
\[
\left\| \mathrm{A}_{U,s}\circ {S_1}\right\|_{L^2(\R^n)} \lesssim_n {
\max\big(  \left\| \mathrm{A}_{V}\right\|_{L^2(\R^n)}, \left\| \M\right\|_{L^2(\R^n)}\big)\eqsim  \left\| \mathrm{A}_{V}\right\|_{L^2(\R^n)}.
}
\]
\end{lemma} 
%%%%%%%%%%%%%%%%%%%%%%%%%%%%%% LEMMA LEMMA LEMMA

%%%%%%%%%%%%%%%%%%%%%%%%%%%%%% PROOF PROOF PROOF
\begin{proof}  Below all the implicit constants in the almost inequality sign may depend on $n$ only, and may differ at each occurrence. The approximate equality in the conclusion is obvious so it suffices to show the estimate
	\[
	\left\| \mathrm{A}_{U,s}\circ {S_1}\right\|_{L^2(\R^n)} \lesssim_n \max\big(  \left\| \mathrm{A}_{V}\right\|_{L^2(\R^n)}, \left\| \M\right\|_{L^2(\R^n)}\big).
	\]
Note also the trivial estimate
\begin{equation}
\label{e:trivest}
\sup_{s\geq 1} \left\| \mathrm{A}_{U,s}\circ {S_1}\right\|_{L^2(\R^n)} \lesssim_n \left\| \M\right\|_{L^2(\R^n)}.
\end{equation}
Indeed the operator on the left hand side has symbol
\[
m(\xi) = \psi\left( \frac{\xi\cdot v}{s} \right) \phi(|\xi|)
\]
where $\phi$ is smooth and supported on an annulus $|\xi|\eqsim 1$. As $|\xi|\eqsim 1\leq s$ on $\mathrm{supp}\, m$ we get that $|\partial^\alpha _\xi m(\xi)|\lesssim |\xi|^{-|\alpha|}$ for $0\leq |\alpha| \leq 10n$, whence the estimate \eqref{e:trivest}.

{For the rest of the proof we assume that $s\leq 1$;} it is convenient to work with the scaled rough averages and maximal operators
\[
\langle f\rangle_{v,s}(x)\coloneqq   \int\displaylimits_{s|t|<1} f (x-t v )  \, \frac{s\d t}{2}, \qquad \mathrm{M}_{V,s} f(x) \coloneqq \sup_{v\in V}  \, \langle |f|\rangle_{v,s}(x) ;
\]
note the somewhat nonstandard use of space scales $1/s$ and that the same scaling property \eqref{e:scale} holds for $\mathrm{M}_V$ in place of $\mathrm{A}_V$.

As $V\subset \mathcal{A}_n(2)$ and $V'\subset \mathbb S^{n-1}$, we have the estimates
\begin{equation}
\label{e:lemmainter-1}
 \left\| \mathrm{M}_{V',\sigma}  \right\|_{L^2(\R^n)} \lesssim   \big\|  \mathrm{M}_{V,\frac{\sigma}{2}}  \big\|_{L^2(\R^n)} \lesssim \big\|  \mathrm{A}_{V,\frac{\sigma}{2}}   \big\|_{L^2(\R^n)}
\end{equation} uniformly over $\sigma>0$,
where the last inequality is obtained by Lemma \ref{l:roughsmooth}. By virtue of \eqref{e:lemmainter-1} it suffices to prove 
\begin{equation} \label{e:lemmainter}
\left\|\mathrm{
 A}_{U,s} \circ S_1 f \right\|_{L^2(\R^n)}\lesssim  \sup_{\sigma<s} \left\| \mathrm{M}_{V',\sigma} (S_1 f) \right\|_{L^2(\R^n)}.
\end{equation}
We write $g\coloneqq S_1 f$ below for simplicity. Let $v(u)\in V$ be chosen such that $|v(u)'-u'|<as$. We claim the estimate
\begin{equation}
\label{e:lemmainter2}
|{A}_{u,s} g| \lesssim  \sum_{k\geq 1} 2^{-kn} \mathrm{M} \big( \l| g| \r_{v(u)',2^{-k}s}\big).
\end{equation}
We first show how \eqref{e:lemmainter2} implies \eqref{e:lemmainter}:  
\[
\begin{split}
\left\|\mathrm{A}_{U,s} g \right\|_{L^2(\R^n)}& \lesssim  \bigg\| \sup_{v'\in V' }\sum_{k\geq 1} 2^{-kn} \mathrm{M} \left( \l |g| \r_{v',2^{-k}s}\right) \bigg\|_{L^2(\R^n)}
\\&\lesssim  \sum_{k\geq 1} 2^{-kn} \left\|  \mathrm{M} \circ \mathrm{M}_{V',2^{-k}s} g\right\|_{L^2(\R^n)} 
 \lesssim  \sup_{k\geq 1} \left\| \mathrm{M}_{V',2^{-k}s} g \right\|_{L^2(\R^n)}
\end{split}
\] 
which is \eqref{e:lemmainter}; note that \eqref{e:lemmainter2} has been used in the first inequality.

In order to prove  \eqref{e:lemmainter2} we may assume, by rotation invariance, that  $v(u)'=(1,0,\ldots, 0)$.  Then  
\begin{equation}
\label{e:bound:multip0}
|u_1|\lesssim 1, \qquad 
\frac{1}{s}\sup_{2\leq j\leq n} |u_j|\leq \frac{4}{s} \sup_{2\leq j\leq n} |u'_j| \leq 4    \min\big({\textstyle\frac1s},a\big).
\end{equation} 
Let us write 
\[
(A_{u,s}{S_1}f)^\wedge\eqqcolon m_{u,s}\widehat f
\]
for the multiplier of $A_{u,s}{S_1}$. Now note that under our assumptions on the parameter $a$ there exists a dimensional constant $c>1$ such that  the set ${R}_{v,  c  s}$ contains the support of $m_{u,s} $. Relying on \eqref{e:bound:multip0},   for every multiindex $\alpha$ we obtain
\begin{equation}
\label{e:bound:multip}
\left|\partial_{\xi}^{\alpha}  m_{u,s}(\xi)\right| = 
\left|\partial_{\xi}^{\alpha} \big[ \widehat{{S_1}}(\xi)\psi\left(s^{-1} u \cdot \xi\right) \big]\right| \lesssim  {s^{-\alpha_1}} \min\big(1,s^{- |\alpha|+ \alpha_1  }\big) {\eqsim s^{-\alpha_1}}.
\end{equation}
{since $s\leq 1$.} This readily implies that $K_{u,s}\coloneqq \widehat{m_{u,s}}$ satisfies the estimate
\[
\left|K_{u,s}(x_1,\ldots, x_{n}) \right| \lesssim   \frac{s}{ (1+s|x_1|)^{10n}}    \bigg(1+\sum_{j=2}^{n}|x_j| \bigg)^{-10n}  
\lesssim \sum_{k\geq 1}^\infty 2^{-kn} \widetilde{\cic{1}}_{B_k}(x),
\]
where $B_k$ is the tube centered at the origin and of sidelengths $s^{-1}2^{k}$ along $e_1$ and $ 2^{k}$ along $e_2,\ldots, e_n$
and $\widetilde{\cic{1}}_{B_k}=  \cic{1}_{B_k}/|B_k|$. Hence
\[
|A_{u,s}g|= |g*K_{u,s}| \leq |g|*|K_{u,s}|    \lesssim \sum_{k\geq 1}^\infty 2^{-kn} |g|* \widetilde{\cic{1}}_{B_k} \lesssim  \sum_{k\geq 1}^\infty 2^{-kn} \mathrm{M} \big( \l| g |\r_{v(u)',2^{-k}s}\big)
\]
where we used the inequality $|g|* \widetilde{\cic{1}}_{B_k}\lesssim\mathrm{M} \big( \l |g| \r_{v(u)',2^{-k}s}\big)$. We have reached \eqref{e:lemmainter2} and the proof is thus complete.
\end{proof}
%%%%%%%%%%%%%%%%%%%%%%%%%%%%%% PROOF PROOF PROOF

Lemma \ref{lem:closeness} tells us that it suffices to study small scales when $f$ has annular frequency support.

%%%%%%%%%%%%%%%%%%%%%%%%%%%%%% COROLLARY COROLLARY COROLLARY
\begin{corollary}  \label{cor:largescale}
There holds \[\sup_{s>2^{-5}} \left\|\mathrm{A}_{\mathcal A_n(2),s}\circ {S_1}   \right\|_{L^2(\R^{n})} \lesssim_n 1.\]
\end{corollary}
%%%%%%%%%%%%%%%%%%%%%%%%%%%%%% COROLLARY COROLLARY COROLLARY

%%%%%%%%%%%%%%%%%%%%%%%%%%%%%% PROOF PROOF PROOF
\begin{proof} Let $s>2^{-5}$. Let $V$ be a $2^{-10}$-net in $\mathbb S^{n-1}
$. Then $U=\mathcal A_n(2)$ and $V$ satisfy the assumption \eqref{e:red:closeness} of Lemma \ref{lem:closeness} with $a=2^{-4}$. Hence
\[  
\left\|\mathrm{A}_{\mathcal A_n(2),s}\circ {S_1}   \right\|_{L^2(\R^{n})} \lesssim_n   \left\|\mathrm{A}_{V}    \right\|_{L^2(\R^{n})}\lesssim_n 1
 \]
as $V$ has cardinality $\Theta_n$ and for every $v$ the averaging operators $f\mapsto  A_{v }f $ are bounded on $L^2(\R^n)$. Note that an alternative proof of the statement of the corollary follows by the kernel estimates established in the proof of Lemma~\ref{lem:closeness}, applied for $s\gtrsim 1$.
\end{proof}
%%%%%%%%%%%%%%%%%%%%%%%%%%%%%% PROOF PROOF PROOF

%%%%%%%%%%%%%%%%%%%%%%%%%%%%%% COROLLARY COROLLARY COROLLARY
\begin{corollary} \label{cor:descent}
Let $0<s< 2^{-5}$, $n\geq 3$, and let $Z\subset  \mathcal{A}_{n}(1)$ have the property that \[\dist(Z,\xi^\perp)\leq a s\] for some $a \leq 2^{3}$. Denote by $\Pi_\xi Z$ the orthogonal projection of $Z$ on $\xi^\perp \equiv \mathbb R^{n-1}$. Then
\[
\sup_{\substack{U\subset Z\\ \# U \leq N} } \left\|\mathrm{A}_{U,s}\circ {S_1} \right\|_{L^2(\R^{n})} \lesssim_n 
 \sup_{\substack{V\subset \Pi_\xi Z\\ \# V \leq N} } \left\|\mathrm{A}_{V } \right\|_{L^2(\R^{n-1})} .
\]
\end{corollary}
%%%%%%%%%%%%%%%%%%%%%%%%%%%%%% COROLLARY COROLLARY COROLLARY

%%%%%%%%%%%%%%%%%%%%%%%%%%%%%% PROOF PROOF PROOF
\begin{proof}Let  $U\subset Z$ have $\# U \leq N$. The set $V\coloneqq\Pi_\xi U$ is contained in $\Pi_\xi Z\subset \R^{n-1}$ and has $\#V\leq N$. 
Then $U$ and $V$ satisfy the assumption \eqref{e:red:closeness} of Lemma \ref{lem:closeness}.  
 We  obtain that 
\[
 \left\|\mathrm{A}_{U,s} \circ {S_1}  \right\|_{L^2(\R^{n})} \lesssim_{n}  
\left\|\mathrm{A}_{V  }   \right\|_{L^2(\R^{n})} 
\lesssim_{n}
\left\|\mathrm{A}_{V }   \right\|_{L^2(\R^{n-1})} 
\]
where the first inequality is an application of   Lemma \ref{lem:closeness}, and the second follows from Fubini's theorem applied on the slices $E_b\coloneqq \{x\in \R^{n}:\, x\cdot \xi =b\},\, b\in \R$. The proof is complete.\end{proof}
%%%%%%%%%%%%%%%%%%%%%%%%%%%%%% PROOF PROOF PROOF

%%%%%%%%%%%%%%%%%%%%%%%%%%%%%% SECTION SECTION SECTION
\subsection{Polynomial partition of directions} \label{ss:ppd} We write $x=(x_1,\ldots,x_n)\in \mathbb \R^n$ and denote by $\mathbb R[x_1,\ldots,x_n]$ the ring of $n$-variable polynomials with real coefficients.

A   \emph{transverse complete intersection} $Z$ in $\R^n$  of dimension $m\in\{1,\ldots, n-1\}$ is the common zero set of polynomials $P_1,\ldots, P_{n-m}\in \R[x_1,\ldots, x_n]$:
\[
Z=\mathbf{Z}(P_1,\ldots,P_{n-m})= \{x\in \R^n:\,P_1(x) = \ldots= P_{n-m}(x)=0\},
\]
with the property that 
\begin{equation}
\label{e:wedge1}
\nabla P_{1}(x) \wedge \cdots \wedge \nabla P_{n-m}(x) \neq 0 
\end{equation}
holds for all $x\in Z$. In particular transverse complete intersections are $m$-di\-men\-sio\-nal smooth submanifolds of $\R^n$. We will in general use the shorthand notation (TCI) for transverse complete intersections.
  
The following proposition, which  is fundamental for the analysis in the current paper, is a form of polynomial partitioning on manifold, inspired by \cite{G1,Guth2} (see \cite[Theorem 2.3]{Zahl2} for a different approach) and adapted to the structure of our problem. It  partitions a finite set of directions $V$ lying on an $m$-dimensional TCI into a boundary component $V_\times$ of points lying arbitrarily close to a controlled number of $(m-1)$-dimensional TCIs, and a cellular component $V_\circ$. The component $V_\circ$ is itself partitioned into connected components of the complement of zero sets of boundedly many polynomials of controlled degree, each containing at most a fraction of the original number of points.     The proof of this proposition together with a  detailed presentation and analysis of relevant tools from algebraic geometry is contained in Section \ref{sec:alggeom}.
 
%%%%%%%%%%%%%%%%%%%%%%%%%%%%%% PROPOSITION PROPOSITION PROPOSITION
\begin{proposition}\label{prop:polpart} Let $Z=\mathbf{Z}(P_1,\ldots,P_{n-m})\subset \R^n$ be a transverse complete intersection of degree $D$. Let  $V\subset Z$ be a finite point set with $\#V\leq N^m$. For each  integer $E$ with $E^{m-1}\geq D^n$   and $\delta>0$ we may perform a partition  
\[
V=V_{\circ} \cup V_{\times}
\]
with the following properties.
\begin{itemize}
\item[\emph{1}.] There exist $\leq \Theta_{m,n} D^n$ transverse complete intersections \[W_j=\mathbf{Z}(P_1,\ldots, P_{n-m},Q_j)\] of dimension $m-1$ and degree $\leq \Theta_{m,n}  E$ such that
\[
\sup_{v \in V_{\times}} \inf_{j} \dist(v,W_j) <\delta.
\]
\item[\emph{2}.]  
 There exist $\leq \Theta_{m,n}D^n E^m$ disjoint connected subsets $\mathsf{C} \in \vec {\mathsf{C}}  $ of $Z$ with the property that
 \[
 V_\circ = \bigcup_{\mathsf{C} \in \vec {\mathsf{C}}} V_{\mathsf{C}}, \qquad V_{\mathsf{C}}\coloneqq V\cap \mathsf{C}, \qquad \#V_{\mathsf{C} }\leq \bigg(\frac N E\bigg)^m,
 \]
and such that 
for almost every $\xi \in \R^n$ and every  $a\in \R$, the intersection $\mathsf{C}\cap \{x\in \R^n:\xi\cdot x=a \}$ is nontrivial for at   most $\leq \Theta_{n,m} D^{2n}E^{m-1}$ elements of $\vec {\mathsf{C}} $.
 \end{itemize}
\end{proposition}
%%%%%%%%%%%%%%%%%%%%%%%%%%%%%% PROPOSITION PROPOSITION PROPOSITION

%%%%%%%%%%%%%%%%%%%%%%%%%%%%%% SECTION SECTION SECTION
\section{Directional averages along arbitrary directions on $\mathbb S^2$} \label{sec:3dgeneraldirs}
This section is dedicated to the proof of Theorem~\ref{t:main:card}. Given a set of directions in $\mathbb S^2$ we will apply polynomial partitioning, in the form of Proposition \ref{prop:polpart}, to reduce the problem of estimating the corresponding maximal directional average into two parts. The first one, corresponding to the cells of the partition, will be handled inductively. The second term corresponds loosely to the zero set of the polynomial partitioning, which in this case is an algebraic variety of dimension $1$. In particular, the application of Proposition~\ref{prop:polpart} of the previous section allows us to work with a nice class of algebraic varieties, namely transverse complete intersections.

%%%%%%%%%%%%%%%%%%%%%%%%%%%%%% SECTION SECTION SECTION
\subsection{Directional averages along one dimensional varieties} 
In the theorem below we work on the particular case of algebraic varieties of dimension one, contained in $\SS^2$. We note here that for a single smooth parametrizable curve in $\mathbb S^{n-1}$ an argument of C\'ordoba from  \cite{Cor1982}*{p. 223}, or of Christ, Duoandikoetxea, and Rubio de Francia from \cite{CDR}, could be used to yield a bound of the form $D \log N$. However these arguments do  not seem to directly provide the degree dependence $D^\frac12$ in the case of a general one-dimensional algebraic variety which might have several connected components. This degree dependence is crucial for our application in the case of averages along directions in $\SS^2$, forcing us to modify the argument from  \cite{Cor1982} in order to get the correct behavior $D^\frac12$  in terms of the degree, at the price of a slightly worse bound in terms of the cardinality of the set of directions. As we will need to apply the result below for a relatively high degree $D$ this trade-off is in our favor. We only insist on this variation for the case of algebraic varieties of dimension $1$ contained in $\SS^2$ as for the higher codimensional case the result of C\'ordoba is sufficient for our purposes; see Subsection~\ref{ss5:1n}. 

%%%%%%%%%%%%%%%%%%%%%%%%%%%%%% THEOREM THEOREM THEOREM
 \begin{theorem} \label{thm:curve}
Let $Z\subset \mathbb  S^2 $ be a transverse complete intersection of dimension $1$, given by a  non-zero polynomial $P$ with $\mathrm{deg}P=D$, namely
\[
Z=\left\{x \in \mathbb \R^3:\,P(x)=P_{\mathsf{sph}}(x)=0\right\}
\] 
where $P_{\mathsf{sph}}(x)\coloneqq x_1^2+x_2^2+x_3^2-1$ and 
\[
\nabla P(x) \wedge x \neq 0 \qquad \forall x \in Z.
\]  
Then 
\[
\sup_{\substack {V\subset Z\\ \# V= N}} 
\big\|   \mathrm{A}_{V}  f  \big\|_{L^2(\R^{3 })} \lesssim   D^\frac12  (\log N )^\frac32   \|f\|_{L^2(\R^{3 })} .
\]
Furthermore we have the following single-annulus estimate
	\begin{equation}\label{eq:curvesa}
	 \sup_{s>0} \sup_{\substack {V\subset Z\\ \# V= N}} \big\|   \mathrm{A}_{V,s}  \circ {S_1} f  \big\|_{L^2(\R^{3 })} \lesssim  D^\frac12  \log N    \|{S_1} f\|_{L^2(\R^{3 })} .
	\end{equation}
\end{theorem}
%%%%%%%%%%%%%%%%%%%%%%%%%%%%%% THEOREM THEOREM THEOREM

%%%%%%%%%%%%%%%%%%%%%%%%%%%%%% PROOF PROOF PROOF
\begin{proof} We fix a transverse complete intersection $Z$ as in the statement of the theorem and let $V\subset Z$ be a set of directions with $\#V=N$. By an application of the Chang-Wilson-Wolff inequality of Proposition~\ref{p:CWW}, the proof of the lemma will follow from the single annulus estimate \eqref{eq:curvesa} which we prove in what follows. By Corollary \ref{cor:largescale}, it suffices to treat the case $0<s<2^{-5}$. The proof is divided in steps.

%%%%%%%%%%%%%%%%%%%%%%%%%%%%%% SECTION SECTION SECTION
\subsubsection*{Decomposition into good and bad components}
The variety $Z$ can be decomposed into connected components $Z=\cup_j Z_j$. It is important to note here that the number of connected components of $Z$ is at most $O(D^2)$;  see for example \cite{BarBas} and the references therein.
	
Let $\Omega\subset \mathbb S^2$ be a $2^{-10}s$-net, that is a set of $2^{-10}s$-separated vectors $\xi\in \mathbb S^2$ such that the collection of caps 
\[
\omega_{\xi,s}=\big\{v\in \mathbb S^2:|v-\xi|<s\big\}, \qquad{\xi\in\Omega},
\]
is a finitely overlapping cover of $\mathbb S^2$. A \emph{cluster} $\mathsf{k}_\xi\subset \{Z_j\}_j$ with \emph{top} $\xi\in \Omega$ is a subset $\{Z_{j_k}\}_k\subseteq \{Z_j\}_j$ with $Z_{j_k}\subseteq R_{\xi,3s}$. We say that a cluster is \emph{bad} if its interior contains (in the sense of set inclusion) more than $D$ distinct components from $\{Z_j\}_j$. By a greedy selection algorithm we can identify a set $\Omega_{\mathsf{bad}}\subseteq \Omega$ and a corresponding set $Z_{\mathsf{bad}}\subseteq\{Z_j\}$ of at most $O(D)$ bad clusters. Setting $\Omega_{\mathsf{good}}\coloneqq \Omega\setminus \Omega_{\mathsf{bad}}$ we then know that for every $\xi \in \Omega_{\mathsf{good}} $ the band $R_{\xi,3s}$ contains at most $D$ connected components from $\{Z_j\}_j$. We define $V_{\mathsf{good}}$ and $V_{\mathsf{bad}}$ to be those directions from $W$ contained in $Z_{\mathsf{good}}$, and $Z_{\mathsf{bad}}$, respectively. We also write $V_{\xi}$ for those directions contained in a cluster $\mathsf{k}_\xi$ for some $\xi\in\Omega$.

%%%%%%%%%%%%%%%%%%%%%%%%%%%%%% SECTION SECTION SECTION
\subsubsection*{Estimate for the bad component directions $V_{\mathsf{bad}}$}
We first estimate the operator corresponding to directions in $V_{\mathsf{bad}}$ on a single frequency annulus. We have
\begin{equation}
\label{e:2dtrick0}
\big\|   \mathrm{A}_{V_{\mathsf{bad}},s}\circ  {S_1}f  \big\|_{L^2(\R^3)}\leq \Bigg (\sum_{\xi\in \Omega_{\mathrm {bad}}} \big\|  \mathrm{A}_{V_\xi,s}\circ {S_1}  f\big\|_{L^2(\R^3)} ^2 \Bigg)^\frac12. 
\end{equation}
We note that for $\xi \in \Omega_{\mathsf{bad}}$ the set $\bigcup\{Z_j:Z_j\in\mathsf{k}_\xi\}$ is contained in a frequency band $R_{\xi,3s}$ so we can approximate the directions in $V_\xi$ for $\xi\in\Omega_{\mathsf{bad}}$ by directions lying on $\xi^\perp\cap \mathbb S^2\equiv \mathbb S^1$. Note that the sets $\{V_\xi\}_{\xi\in\Omega_{\mathsf{bad}}}$ are pairwise disjoint, by our greedy selection algorithm, and each has cardinality $\leq N$. More precisely, we apply Corollary~\ref{cor:descent} with $\xi\in \Omega_{\mathrm{bad}}$ to obtain
\begin{equation}
\label{e:2dtrick}
\begin{split}
& \big\|  \mathrm{A}_{V_\xi,s}\circ {S_1}  f\big\|_{L^2(\R^3)} \lesssim\bigg(  \sup_{\substack{U\subset \xi^\perp\cap \mathbb S^2\\ \#U\leq N } } \|\mathrm{A}_{U }\|_{L^2(\R^2)}\bigg)\|S_1 f\|_{L^2(\R^3)} 
\\
&\qquad \lesssim \sqrt{\log N} \|S_1 f\|_{L^2(\R^3)}.
\end{split}
\end{equation}
The last inequality follows by an application of the two-dimensional {single-scale} result of Katz \cite{KB}. Alternative proofs of this two-dimensional result have since been given in \cite{Dem,DPGTZK}.
Combining  \eqref{e:2dtrick0} with \eqref{e:2dtrick} we gather that 
\[
\big\|   \mathrm{A}_{V_{\mathsf{bad}},s}\circ  {S_1}f  \big\|_{L^2(\R^3)} \lesssim  \bigg (\sum_{\xi\in \Omega_{\mathrm {bad}}}  \log N 
\bigg)^\frac12 \|{S_1} f\|_{L^2(\R^3)}\lesssim D^\frac12 \sqrt{\log N}\|{S_1} f\|_{L^2(\R^3)}
\]
since  $\# \Omega_{\mathsf{bad}}\leq D$. This gives the desired estimate for the bad part. 

\subsubsection*{Estimate for  $V_{\mathsf{good}}$}
We now move to the estimation of $V_{\mathsf{good}}$ which relies  	on overlap considerations. Let $Z_{\mathsf{good}}$ consist of some connected components   $\{Z_j\}_{j\in J}$ and let $V_j$ denote the directions of $V$ contained in $Z_j$. We write
\[
V_j= \{v_{j,1},\ldots,v_{j,n_j}\}
\]
where, for fixed $j\in J$, the directions are sorted in consecutive order.

In this proof, we need a rough analogue of $A_{v,s}$. Let $I\subset \R $ be any interval and for $g \in L^2(\R^n)$ define
\[
B_{v,I}g(x) \coloneqq \int_{\R^n} \widehat g(\beta) \cic{1}_I(\beta \cdot v) \e^{ix\cdot\beta} \d\beta, \qquad \mathrm{B}_{V,I} g \coloneqq \sup_{v\in V} |B_{v,I} g|.
\] 
We also define $I_{a,\sigma}\coloneqq (a-\sigma, a+\sigma)$. By an averaging argument we have that
\[
\| \mathrm{A}_{V_{\mathsf{good}},s}\circ  {S_1} f\|_{L^2(\R^n)} \lesssim  \|\psi\|_{\mathrm {BV}} \sup_{0<\sigma\leq s} \sup_{|a|<s} \| \mathrm{B}_{V_{\mathsf{good}, I_{a,\sigma}}}(S_1f)\|_{L^2(\R^n)} .
\]
Therefore, it suffices to estimate the $L^2(\R^3)$-norm of the operator $ \mathrm{B}_{V_{\mathsf{good}, I_{a,\sigma}}}\circ  {S_1}$ uniformly over $0<\sigma\leq s$. This can be done as follows. Set
\[
{L}_{\beta,a,\sigma} \coloneqq \{v\in \mathbb S^2:\,|\beta\cdot v- a|<\sigma\}{=\{v\in\mathbb S^2:\, \beta\cdot v \in I_{\alpha,\sigma}}\}, \qquad \beta \in \mathcal A_n(1).
\]
Then 
\[
\begin{split} &
\quad \| \mathrm{B}_{V_{\mathsf{good}, I_{a,\sigma}}}\circ  {S_1}\|_{L^2(\R^3)}\leq \sup_{\beta \in \mathcal A_{n}(1)}\bigg(\sum_{j\in J}\sum_{\ell=1} ^{n_j-1} \big|\ind_{L_{\beta,a,\sigma}}( v_{j,\ell+1}) - \ind_{L_{{\beta},a,\sigma}}( v_{j,\ell}) \big|^2 \bigg)^\frac12 
\\
& + \sup_{\beta \in \mathcal A_{n}(1)}\bigg(\sum_{j\in J}  \big|
\ind_{L_{\beta,a,\sigma}}( v_{j+1,1}) - \ind_{L_{{\beta},a,\sigma}}( v_{j,n_j}) \big|^2 \bigg)^\frac12 + \|  \mathrm{B}_{\widetilde{V}_{\mathsf{good}}, I_{a,\sigma}}\circ  {S_1}\|_{ L^2(\R^3) }
\end{split}
\]
where $\widetilde V_{\mathsf{good}}\subset V_{\mathsf{good}}$ satisfies $\# \tilde V_{\mathsf{good}} \leq \#V_{\mathsf{good}}/2$. 

We now estimate the square functions in the first two summands. For this fix some $\beta\in \mathcal A_{n}(1)$ and pick $\xi\in \Omega$ such that $\beta ' \in \omega_{\xi,s}${. Here we remember that $\beta'=\beta/|\beta|$ and} note that $\xi$ remains fixed as long as $\beta$ is fixed.
For the square function in the first summand note that if $\ind_{L_{\beta,a,\sigma}}( v_{j,\ell+1}) - \ind_{L_{\beta,a,\sigma}}( v_{j,\ell}) \neq 0$ then the piece of $Z_j$ between $v_{j,\ell}$ and $v_{j,\ell+1}$, which is connected, must cross one of the hyperplanes $\{v:\,\beta\cdot v  =a \pm \sigma\}$. Overall the variety $Z$ can cross these hyperplanes at most $O(D)$ times so the first summand above is $O(D^\frac12)$.

 For the square function in the second summand we note that $\ind_{L_{\beta,a,\sigma}}( v_{j+1,1}) - \ind_{L_{\beta,a,\sigma}}( v_{j,n_j})\neq 0 $ implies one of the two following possibilities
\[
\begin{split}
&\text{either}\qquad |\beta \cdot v_{j+1,1} - a|{<} \sigma \quad \text{and}\quad |\beta \cdot v_{j,n_j}-a|{\geq  \sigma},
\\
&\text {or} \qquad |\beta \cdot v_{j+1,1}-a|{\geq\sigma} \quad \text{and}\quad |\beta\cdot v_{j,n_j}-a|{<}  \sigma.
 \end{split}
\]
In either one of the cases above there is at least one component $Z_*\in Z_{\mathrm {good}}$ and $\beta' \in \mathbb S^2$ such that $|\beta' \cdot v-a|\leq {s}$ for some $v\in Z_*$, {where we crucially use that $|a|,\sigma\leq s$ and that $\beta\in \mathcal A(1)$}. This implies that $Z_* \cap R_{\xi,3s}\neq \varnothing$ {using the triangle inequality and the facts that $|a|\leq s$ and $\beta'\in \omega_{\xi,s}$}. Thus a bound for the second summand can be calculated by counting how many of the components $Z_*\in Z_{\mathsf{good}}$ can intersect $R_{\xi,3s}$. As the components $Z_*\in Z_{\mathsf{good}}$, at most $D$ of them can be  contained in the interior of $R_{\xi,3s}$. On the other hand the components $Z_*$ that satisfy $Z_* \nsubseteq R_{\xi,3s}$ and $Z_*\cap R_{\xi,3s}\neq \varnothing$  can be counted by the number of connected components of the set $ \{\eta\in\mathbb S^2:\, \eta\cdot \xi=\pm 3 s\} \setminus Z $ which is $O(D)$. This follows from the fact that the set $ \{\eta\in\mathbb S^2:\, \eta\cdot \xi=\pm 3s\} $ is a smooth algebraic variety of dimension $1$ and an application of the Milnor-Thom theorem;  see for example \cite{BarBas} and the references therein. 

We have thus proved the estimate
\[
\Big\| \mathrm{B}_{V_{\mathsf{good}, I_{a,\sigma}}} \circ S_1\Big\|_{L^2(\R^3)}\leq CD^\frac12 +\Big\| \mathrm{B}_{\widetilde{V}_{\mathsf{good}, I_{a,\sigma}}}\circ S_1 \Big\|_{L^2(\R^3) }.
\]
This recursive inequality can be iterated on $V_{\mathrm{good}}$ yielding the bound $D^\frac12 \log N$ for the good part so the proof is complete.\end{proof}
%%%%%%%%%%%%%%%%%%%%%%%%%%%%%% PROOF PROOF PROOF

%%%%%%%%%%%%%%%%%%%%%%%%%%%%%% SECTION SECTION SECTION
\subsection{Reduction to single annulus and polynomial partition}
Theorem~\ref{t:main:card}  is reduced via  Proposition \ref{p:CWW} and Corollary \ref{cor:largescale}   to the following single annulus bound. As mentioned after the statement of Theorem~\ref{t:main:card} in \S\ref{sec:intro}, standard counterexamples show that the growth rate in \eqref{e:mainsa} is sharp up to the iterated logarithmic term.

%%%%%%%%%%%%%%%%%%%%%%%%%%%%%% PROPOSITION PROPOSITION PROPOSITION
\begin{proposition} \label{p:single:card}  Let $V\subset \mathbb S^{2}$ be a finite subset with $\#V=N^2$. Then  for all $k\geq 1$
\begin{equation} \label{e:mainsa}
\sup_{s>0} \left\|
\mathrm{A}_{V,s}\circ {S_1}   \right\|_{L^2(\R^{3})}    \lesssim_k N^{\frac12}  \log^{[k]}(N)   .
\end{equation}
\end{proposition}
%%%%%%%%%%%%%%%%%%%%%%%%%%%%%% PROPOSITION PROPOSITION PROPOSITION

%%%%%%%%%%%%%%%%%%%%%%%%%%%%%% PROOF PROOF PROOF
 As anticipated, the proof uses polynomial partitioning on $\mathbb S^2$.   We will prove the proposition by showing inductively  that  
\[
K_N \coloneqq   \sup_{s>0} \sup_{\substack{V\subset \mathbb{S}^2 \\ \#V = N^2}} \left\| \mathrm{A}_{V,s}\circ {S_1}  \right\|_{L^2(\R^{3})}^2 
\]
 fulfills \eqref{e:mainsa}. More precisely, for every positive integer $E>1$ we will prove the recursive estimate
\begin{equation} \label{e:rec}
 K_N  \leq   \Theta E \left[ (\log N )^3+  K_{\frac NE} \right]
\end{equation}
where the constant $\Theta>0$ is independent of $N$ and $E$.    We show at the end of the proof how \eqref{e:rec} implies the proposition.

%%%%%%%%%%%%%%%%%%%%%%%%%%%%%% PROOF PROOF PROOF
\begin{proof}[Proof of \eqref{e:rec}] By Corollary \ref{cor:largescale} it suffices to argue for $0<s<2^{-5}$. Throughout, $\Theta>0$ will be some absolute constant which may vary from line to line.  To prove \eqref{e:rec} we fix some  $V\subset \mathbb S^{2}$ with cardinality at most $N^2$. Applying the polynomial partition on $\mathbb S^2=\mathbf{Z}(P_{\mathsf{sph}})\subset \R^3$ from Proposition~\ref{prop:polpart},  so that $D=2$, with the choice $\delta=2^{-10}s$,  we obtain a decomposition
\begin{equation}
\label{e:pf320}
V = V_{\circ} \cup V_{\times}
\end{equation}
with the following properties. First,
there exist $\leq \Theta $ transverse complete intersections $W_j=\mathbf{Z}(P_{\mathsf{sph}},Q_j)$ of dimension $1$ and degree $\leq \Theta  E$ such that
\begin{equation}
\label{e:pf321}
\sup_{v \in V_{\times}} \inf_{j} \dist(v,W_j) <2^{-10}s.
\end{equation}
Secondly, there exist $\leq \Theta  E^2$ disjoint connected subsets of  $\mathsf{S}^2$, denoted by $\mathsf{C} \in \vec {\mathsf{C}}  $,   with the property that
\begin{equation}
\label{e:pf322}
 V_\circ = \bigcup_{\mathsf{C} \in \vec {\mathsf{C}}} V_{\mathsf{C}}, \qquad V_{\mathsf{C}}\coloneqq V_{\circ}\cap \mathsf{C}, \qquad \#V_{\mathsf{C} }\leq \left(\frac N E\right)^2,
 \end{equation}
and such that  for almost every $\xi \in \R^2$ and every  $a\in \R$, the intersection $\mathsf{C}\cap \{x\in \R^2:\xi\cdot x=a \}$ is nontrivial for   $\leq \Theta E$ elements of $\vec {\mathsf{C}} $.
 
%%%%%%%%%%%%%%%%%%%%%%%%%%%%%% REMARK REMARK REMARK
\begin{remark} For those readers who have not looked ahead at the proof of Proposition \ref{prop:polpart}, we provide an explanation of how the bounded crossing property
\begin{equation}
\label{e:bcp}
\# \big\{\mathsf{C} \in \vec {\mathsf{C}}:\,\mathsf{C} \cap \{x\in \R^2:\xi\cdot x=a \} \neq \varnothing\big\}\leq \Theta E
\end{equation}
is obtained. This explanation is fully rigorous if the points of $V$ initially lie on a  coordinate neighborhood of $\mathbb S^2$ (say the $1/10$ neighborhood of the north pole), and is made fully rigorous in a more general setup in the proof of  Proposition \ref{prop:polpart}. The cells $\mathsf{C} \in \vec {\mathsf{C}}$ are the connected components of $\mathbb S^2\setminus Z$, where $Z$ is an algebraic subvariety of $\mathbb S^2$ obtained as the  zero set $Z=\mathbf{Z}(P, P_{\mathsf{sph}})$ with $P$ a polynomial of degree $\leq \Theta E$. The intersection of a generic hyperplane with $\mathbb S^2\setminus Z$ can have at most $\leq \Theta E$ connected components (as many as the common zeros of $P, P_{\mathsf{sph}}$ on the hyperplane), and the number of these connected components coincides with the left hand side of the last display.
 \end{remark}
%%%%%%%%%%%%%%%%%%%%%%%%%%%%%% REMARK REMARK REMARK

%%%%%%%%%%%%%%%%%%%%%%%%%%%%%% SECTION SECTION SECTION
\subsubsection*{Controlling $V_\times$}
In view of  \eqref{e:pf321} we may find subsets $V_j\subset W_j$ with $\#V_j \leq N^2$ so that for each  $v\in V_\times$ there exists $j$ and $w(v)\in V_j$ with $|v-w(v)|<2^{-10}s$. Hence using Lemma  \ref{lem:closeness},  directions from $V_\times$ can be effectively approximated with directions from $ V_j\subset W_j$. More precisely,  we have that 
\begin{equation}\label{e:polybound}
\begin{split}
\big\|  \mathrm{A}_{V_\times,s} \circ {S_1} f  \big\|_{L^2(\R^3)}^2 & \leq  \Theta \left(\sup_{j}   \big\|  \mathrm{A}_{V_j}   \big\|_{L^2(\R^3)}^2\right) \left\|  {S_1} f  \right\|_{2}^2\\ &  \leq   \Theta E(\log N)^3 \left\|{S_1}f\right\|_{L^2(\R^3)}^2 
\end{split}
\end{equation}
applying Lemma \ref{lem:closeness} for the first inequality and   Theorem~\ref{thm:curve} for each subset $V_j$ of the  transverse complete intersection $W_j$ to finish.

%%%%%%%%%%%%%%%%%%%%%%%%%%%%%% SECTION SECTION SECTION
\subsubsection*{Splitting $V_\circ$ into good and bad clusters}
First of all, let $\xi \in \Omega$ be a $2^{-10}s$-net on $\mathbb{S}^2$ chosen so that  the bounded crossing property \eqref{e:bcp} holds with $a=\pm 3s$, for \emph{all} $\xi \in \Omega$.    We define $\mathsf{k}_\xi\subset \vec{\mathsf{C}}$ to be a cluster of cells with top $\xi \in \Omega$ if
\[
\mathsf{C} \subset R_{\xi,3s} \qquad \forall \mathsf{C}\in {\mathsf{k}}_\xi.
\]
We say that a cluster ${\mathsf{k}}_\xi$ is bad if it contains more than $E$ cells. As $\# \vec{\mathsf{C}}\leq \Theta E^2$, by  an iterative selection algorithm, we may construct $\Omega_{\mathsf{bad}}\subset \Omega$ with   $\#\Omega_{\mathsf{bad}}\leq \Theta E$ such that setting
\[
\vec{\mathsf{C}}_{\mathsf{good}} \coloneqq \vec{\mathsf{C}} \setminus \bigcup_{\xi \in\Omega_{\mathsf{bad}} }\mathsf{k}_\xi
\]
there holds
\begin{equation}
\label{e:nottoomuchcont}
\#\left\{\mathsf{C}\in \vec{\mathsf{C}}_{\mathsf{good}}:\,\mathsf{C} \subset R_{\xi,3s} \right\} \leq  E \qquad 
\forall \xi \in \Omega.
\end{equation}
Accordingly we define
\[
\begin{split} &
V_\xi  \coloneqq \bigcup_{\mathsf{C}\in \mathsf{k}_\xi} V_{\mathsf{C}}
,\qquad 
V_{\circ, \mathsf{bad}}  \coloneqq \bigcup_{\xi \in\Omega_{\mathsf{bad}} }V_{\xi},
\qquad
V_{\circ, \mathsf{good}}   \coloneqq\bigcup_{\mathsf{C}\in \vec{\mathsf{C}}_{\mathsf{good}}} V_{\mathsf{C}}. 
\end{split}
\]

%%%%%%%%%%%%%%%%%%%%%%%%%%%%%% SECTION SECTION SECTION
\subsubsection*{Controlling $ V_{\circ, \mathsf{bad}}$} Each of the $\leq\Theta E$ bad clusters will be controlled by a lower dimensional estimate. By construction, for each $v\in V_\xi$ we may find $u(v) \in W_\xi \coloneqq \SS^{2} \cap \xi^\perp$ with $|v-u(v)|\leq 3s$. Define $U_\xi= \{u(v):v\in V_\xi\}\subset W_\xi$, and note that $\#U_\xi\leq N^2$. Using    Lemma  \ref{lem:closeness} again we get
\begin{equation}\label{e:badbound}
\big\|  A_{V_{\circ, \mathsf{bad}},s} \circ {S_1} f  \big\|_{L^2(\R^3)}^2  \leq  \Theta  E \left( \sup_{ \xi \in\Omega_{\mathsf{bad}
}}  \big\|  A_{U_\xi }   \big\|_{L^2(\R^3)}^2  \right) \left\| {S_1} f  \right\|_{2}^2 \leq   \Theta E (\log N)^3\left\|{S_1}f\right\|_2^2
\end{equation}
where we used the $D=1$ case of  Theorem \ref{thm:curve} on each $U_\xi \subset W_\xi$.

%%%%%%%%%%%%%%%%%%%%%%%%%%%%%% SECTION SECTION SECTION
\subsubsection*{Controlling $ V_{\circ, \mathsf{good}}$} It is convenient to write $F\coloneqq S_1f$ below.
For $\mathsf{C} \in \vec{\mathsf{C}}_{\mathsf{good}}$ we define  
 \begin{equation}
\label{e:rc}
R_{\mathsf{C}} \coloneqq\bigcup_{v\in  V_{\mathsf{C}}} R_{v,s}.
\end{equation}
We rely on the observation below \eqref{e:bandsR} to gather that  $A_{v,s} F= A_{v,s} F_{R_{\mathsf{C}}}$  whenever $v\in \mathsf{C}$, where we remember that $F_R$ denotes the rough frequency restriction of $F$ to some $R\subset \R^3$. Then we compute the overlap of the bands $\{R_{\mathsf{C}}:\,\mathsf{C} \in \vec{\mathsf{C}}_{\mathsf{good}}\}$ at a generic  $\beta\in \mathbb R^3$.
 By homogeneity it is enough to take $\beta \in \SS^2$. Choose  $\xi \in \Omega$ with $|\xi -\beta|<2^{-10}s$.   Then the overlap of the frequency supports  $R_{\mathsf{C}}$ at $\beta$  is bounded by 
 \begin{equation}
\label{e:rc2}
 \#\Big\{\mathsf{C} \in \vec{\mathsf{C}}_{\mathsf{good}}:\,\mathsf{C}\cap R_{\beta, s}\neq \varnothing\Big\}\leq \#\Big\{\mathsf{C} \in \vec{\mathsf{C}}_{\mathsf{good}}:\,\mathsf{C}\cap R_{\xi, 2s}\neq \varnothing\Big\} .\end{equation}
As by construction of ${\mathsf{C}}_{\mathsf{good}}$
\[
\#\Big\{\mathsf{C}\in \vec{\mathsf{C}}_{\mathsf{good}}:\,\mathsf{C} \subset R_{\xi,3s} \Big\} \leq E 
\]
we are left with counting 
\[
\begin{split} & \quad 
\#\Big\{\mathsf{C}\in \vec{\mathsf{C}}_{\mathsf{good}}:\,\mathsf{C} \cap R_{\xi,2s} \neq \varnothing, \,\mathsf{C} \not \subset R_{\xi,3s}  \Big\}
\\
& \leq  \#\Big\{\mathsf{C}\in \vec{\mathsf{C}}_{\mathsf{good}}:\,\mathsf{C} \cap \left\{x\in \R^2:\xi\cdot x=\pm 3 s   \right\} \neq \varnothing\Big\} \leq \Theta E
\end{split}
\]
where \eqref{e:bcp} has been taken into account. 
As each $V_\mathsf{C}$ contains at most $(N/E)^2$ directions, see \eqref{e:pf322}, and $F_{R_{\mathsf C}}= S_1 F_{R_{\mathsf C}} $,
we note that
\[
\left\|  \mathrm{A}_{   V_\mathsf{C},s}F_{R_{\mathsf{C}}}\right\|^2_{L^2(\R^3)}  \leq  K_{\frac NE} \left\|  F_{R_{\mathsf{C}}}\right\|^2_{L^2(\R^3)}.
\]
Therefore 
\begin{equation} \label{e:cell}
\begin{split} 
\left\|   \mathrm{A}_{  V_{\circ, \mathsf{good}} ,s} F\right\|^2_{L^2(\R^3)} &\leq  \sum_{\mathsf{C} \in \vec{\mathsf{C}}_{\mathsf{good}}} \left\|  \mathrm{A}_{   V_\mathsf{C},s } F\right\|^2 _{L^2(\R^3)}  = \sum_{\mathsf{C} \in \vec{\mathsf{C}}_{\mathsf{good}}} \left\|  \mathrm{A}_{   V_\mathsf{C},s}F_{R_{\mathsf{C}}}\right\|^2_{L^2(\R^3)} 
 \\ 
 &\leq  K_{\frac NE}  \sum_{\mathsf{C} \in \vec{\mathsf{C}}_{\mathsf{good}}} \left\|    F_{R_{\mathsf{C}}}\right\|^2_{L^2(\R^3)} \leq    K_{\frac NE}   \bigg(\sup_{\beta\in \mathbb S^{2}} \sum_{\mathsf{C} \in \vec{\mathsf{C}}_{\mathsf{good}}} \cic{1}_{R_{\mathsf{C}}}(\beta) \bigg) \|F\|_{L^2(\R^3)}^2 \\ & \leq \Theta E K_{\frac NE}   \|F\|_{L^2(\R^3)}^2.
\end{split}
\end{equation} 
Collecting \eqref{e:pf320}, \eqref{e:pf322} with estimates \eqref{e:polybound}, \eqref{e:badbound} and \eqref{e:cell} completes the proof of \eqref{e:rec}. \end{proof}
\begin{proof}[Proof that \eqref{e:rec} implies Proposition \ref{p:single:card}] Observe that an application of \eqref{e:rec} with $E \simeq N/(\log N)^3$ gives the recursive estimate
	\[
	\frac{K_N}{ N } \leq \Theta +\Theta\frac{K_{(\log N)^3}}{(\log N)^3}
	\]
for some numerical constant $\Theta>0$. Iterating we get for any integer $k\geq 1$ that
	\[
	\frac{K_N}{ N } \lesssim \Theta^k +\Theta^k\frac{K_{(\log^{[k]} N)^3 }}{(\log^{[k]} N)^3 }.
	\]
However for any integer $M\geq 1$ we have the trivial estimate $K_{M}\lesssim M^2$ which together with the recursive estimate above gives
\[
K_N \lesssim \Theta^k N + \Theta^k(\log^{[k]}N)^3 N
\]
and completes the proof.
\end{proof}
%%%%%%%%%%%%%%%%%%%%%%%%%%%%%% PROOF PROOF PROOF

%%%%%%%%%%%%%%%%%%%%%%%%%%%%%% SECTION SECTION SECTION
\section{Notions from algebraic geometry, projection varieties, and polynomial partition}\label{sec:alggeom}
In this section, we turn to setting up the algebraic geometry definitions and  tools   needed in the proof of Theorem \ref{t:main:manifold}, beyond the notion of transverse complete intersection already introduced in Subsection \ref{ss:ppd}. We will also prove the Proposition \ref{prop:polpart}, namely the version of polynomial partition used throughout the paper. A short summary that can be used in parallel with reading the proof of Theorem \ref{t:main:manifold}  is the following:
\begin{itemize}
\item[\S\ref{ss:alggeom1}] definition and basic notions concerning   $m$-dimensional real algebraic varieties in $\R^n$;
\item[\S\ref{ss:alggeom2}] approximation of  $m$-dimensional real algebraic varieties with transverse complete intersections; approximate projection of a transverse complete intersection on a lower dimensional subspace;
\item[\S\ref{ss:alggeom3}] approximate polynomial partitioning in $\R^n$ and proof of Proposition \ref{prop:polpart}.
\end{itemize}

%%%%%%%%%%%%%%%%%%%%%%%%%%%%%% SECTION SECTION SECTION
\subsection{Real algebraic varieties and dimension} \label{ss:alggeom1}
 For us a real algebraic variety in $\R^n$ is a set 
 \[
 \mathbf{Z}(I) \coloneqq \{x\in \R^n :\,P(x) =0 \quad \forall P \in I\}
 \]
where $I\subset \mathbb R[x_1,\ldots,x_n]$  is an  ideal. 
We possibly write $\mathbf{Z}(I) =\mathbf{Z}(P_1,\ldots, P_{J})$ if  $\{P_1,\ldots, P_J\}\subset  \mathbb R[x_1,\ldots,x_n] $ is a basis (finite generating set) for the ideal $I$.  

%%%%%%%%%%%%%%%%%%%%%%%%%%%%%% SECTION SECTION SECTION
\subsubsection*{Count and degree} Following the approach of \cite[Section 4]{Sheffer}, we write $\mathsf{ct}\, \mathbf{Z}(I)$ to denote the \emph{count} of $\mathbf{Z}(I)$, that is least cardinality of a basis of the ideal $I$, and $\mathsf{deg}\, \mathbf{Z}(I)$ to denote the least integer $D$ such that there exist polynomials $P_1,\ldots, P_J$ with $  \mathrm{deg} \,P_j \leq D$ for all $j$,  that generate the ideal $I$.  

%%%%%%%%%%%%%%%%%%%%%%%%%%%%%% SECTION SECTION SECTION
\subsubsection*{Dimension}The dimension $\mathsf{dim} \,  Z $ of a real algebraic variety $Z=\mathbf{Z}(I)$ is defined as in \cite[Def.\ 2.8.11]{BCR}; see also \cite{Zahl}.  Let  $\{P_1,\ldots, P_J\}\subset  \mathbb R[x_1,\ldots,x_n] $ be a basis  for the ideal $I$. If $\mathbf{Z}(I)$ has dimension $1\leq m\leq n$, we say that $x\in \mathbf{Z}(I)  $ is a \emph{smooth point}  if 
\[
\mathrm{rank}( \mathrm{D} \vec{P})(x) = n-m,
\]
where $\mathrm{D}\vec P $ is the Jacobian matrix of the vector valued function $\vec{P}:\,\R^{n}\to \R^{J}$, namely the $J\times n$ matrix whose $j$-th row is the vector $\nabla P_j=(\partial_{x_1} P_j, \ldots, \partial_{x_n} P_j) $. We say that $Z=\mathbf{Z}(I)$ is a \emph{smooth algebraic variety} of dimension $1\leq m\leq n$ if every point $x\in Z$ is a smooth point in dimension $m$.
The set of smooth points of an $m$-dimensional algebraic variety in $\R^n$ will be denoted by $Z^{\circ m}$.
A special important case of smooth $m$-dimensional algebraic varieties in $\R^n$ are the transverse complete intersections $Z=\mathbf{Z}(P_1,\ldots, P_{n-m})$ defined in \S\ref{ss:ppd}. 

We conclude our discussion of dimension of algebraic varieties by  recalling the following standard result; see \cite[Theorem 4.8]{Sheffer} for a proof.

%%%%%%%%%%%%%%%%%%%%%%%%%%%%%% LEMMA LEMMA LEMMA
\begin{lemma}[smooth components] \label{lemma:smoothcomp}Fix positive integers $J,D$ and $m \leq n$. There exists a constant $\Theta=\Theta_{J,D,m,n}$ such that the following holds. Let $Z=\mathbf Z(I)\subset \R^n$ be a real algebraic variety with \[\mathsf{ct}\, Z=J,\qquad \mathsf{deg}\, Z=D, \qquad \mathsf{dim}\, Z=m.\]  Then $Z'\coloneqq Z\setminus  Z^{\circ m}$ is a  real algebraic variety with $\mathsf{ct}\, W\leq \Theta,\, \mathsf{deg}\, W\leq \Theta$, and $\mathsf{dim}W<m$.
\end{lemma}
%%%%%%%%%%%%%%%%%%%%%%%%%%%%%% LEMMA LEMMA LEMMA

%%%%%%%%%%%%%%%%%%%%%%%%%%%%%% SECTION SECTION SECTION
\subsubsection*{Classes of algebraic varieties} Finally, it will be convenient to us to introduce  pieces of notation for the classes of algebraic varieties involved in our theorems, namely:
\begin{itemize} 
\item[$\diamond$] $\mathcal Z_{m,n}(D,J)$ is the collection of real algebraic varieties $Z$ in $\R^n$ with 
$\mathrm{dim} \, Z =m, \, \mathrm{deg} \, Z\leq  D, \mathrm{ct} \, Z\leq  J$.
\item[$\diamond$] $\mathcal Z^\times_{m,n}(D)$ is the collection of transverse complete intersections $Z$ in $\R^n$ with 
$\mathrm{dim} \, Z =m, \, \mathrm{deg} \, Z\leq  D$. This is the collection appearing in the quantitative estimate \eqref{e:main:uniform} of Theorem \ref{t:main:manifold}.
\end{itemize}

%%%%%%%%%%%%%%%%%%%%%%%%%%%%%% SECTION SECTION SECTION
\subsection{Approximation by and projection of transverse complete intersections} \label{ss:alggeom2} The ultimate aim of this subsection is to show that any real algebraic variety of dimension $m$ can be approximated in a suitable sense by a controlled number of transverse complete intersections of dimension up to $m$ and of controlled degree.

We begin with detailing a continuity property of   nonsingular zero sets, which are  semialgebraic sets; see \cite{BCR}. It is convenient to introduce some pieces of notation before the statement. Given $1\leq c\leq n$ and polynomials $\{P_j:1\leq j \leq c\}\subset \R[x_1,\ldots, x_n]$ we write
$\vec P=(P_1,\ldots, P_c) :\,\R^n \to \R^c$. Let $\tau$ be an indexing of the ${n}\choose{c}$ choices of $c\times c$ minors $\mathrm{D}_\tau \vec P  $ of $\mathrm{D} \vec P$ and  for each $\tau$ we  write $\Delta_\tau \vec{P} \coloneqq \mathrm{det} \,\mathrm{D}_\tau \vec P $, which is a polynomial of degree at most $D^{c}$. In the statements below the notion of distance $\dist(U,V)$ is as in \S\ref{sec:not}.

%%%%%%%%%%%%%%%%%%%%%%%%%%%%%% LEMMA LEMMA LEMMA
\begin{lemma} \label{lemma:approx} Let $P_1,\ldots,P_c\in \mathbb R[x_1,\ldots,x_n]$ and $Z=\mathbf{Z}(P_1,\ldots, P_c)$. For $\rho,R>0$, let
\begin{equation}
\label{e:urhor}
U_{\rho, R} \coloneqq \Big\{x\in Z :\, |x|< R, \, \max_{\tau} |\Delta_\tau \vec{P} | \geq \rho \Big\}.
\end{equation}
For $\vec \alpha=(\alpha_1,\ldots,\alpha_c)\in \mathbb R^c$, $\vec \beta=(\beta_1,\ldots,\beta_{n-1})\in \R^{n-1}$  define the algebraic varieties
\[
\begin{split}
&  V_{\vec\alpha}\coloneqq\mathbf{Z}(P_1+\alpha_1, \ldots, P_c+\alpha_c),
\\
& W_{\vec\beta}\coloneqq\mathbf{Z}(P_{1,\vec \beta},P_2,\ldots, P_c),  \\ & \textrm{\emph{where}}\; P_{1,\vec{\beta}}(x_1,\ldots,x_n) \coloneqq P_1(x_1 + \beta_1 x_n, \ldots,x_{n-1}+\beta_{n-1}x_n, x_n).
\end{split}
\] 
Then for all $\eps>0$ there exists $\delta=\delta(\eps,\rho,R, \vec{P})>0$ such that
\begin{equation}
\label{e:implicit}
|\vec \alpha| +  |\vec \beta|<\delta\implies 
\mathrm{dist}(U_{\rho,R},V_{\vec\alpha}) + \mathrm{dist}(U_{\rho,R},W_{\vec\beta}) <\eps.
\end{equation}
In addition, if $Z$ is a transverse complete intersection, so are $V_{\vec\alpha}$ and $W_{\vec\beta}$. 
\end{lemma}
%%%%%%%%%%%%%%%%%%%%%%%%%%%%%% LEMMA LEMMA LEMMA

%%%%%%%%%%%%%%%%%%%%%%%%%%%%%% PROOF PROOF PROOF
\begin{proof} We prove the approximation claim for $W_{\vec\beta}$, a similar but simpler proof works for $V_{\vec\alpha}$. Pick a point $\bar x\in U_{\rho,R}$.
By symmetry, suppose that the $c\times c$ matrix $A\coloneqq \{\partial_{x_j} P_{k}(\bar x):\, 1\leq j,k\leq c\}$  is invertible and $|\mathrm{det}\, A|\geq \rho$. Let $F=F(x,\vec\beta):\R^{n} \times \R^{n-1} \to \R^c$ be the vector valued function with $F_1(x,\vec\beta) \coloneqq R_{1,\vec \beta}(x) $ and $F_k(x)\coloneqq  P_k (x) $ for $2\leq k \leq c $. Then $A$ is a $c\times c$ invertible minor of $DF(\bar x,0)$. By the implicit function theorem, for $1\leq j\leq c$ we find smooth functions $h_j=h_j(x_{c+1},\ldots,x_n,\vec\beta)$ defined for $|\vec \beta|+ \sup_{j>c} |x_{j}- \bar x_j|$ sufficiently small such that 
\[ 
\begin{split}
&h_j(\bar x_{c+1},\ldots,\bar x_n,\vec 0_{\R^{n-1}})= \bar x_j, \\ &  R_{1,\vec\beta}\big( h_1(x_{c+1},\ldots,x_n,\vec\beta), \ldots,h_c(x_{c+1},\ldots,x_n,\vec \beta), {x}_{c+1},\ldots,{x}_n\big)=0.
\end{split}
\]
Then \eqref{e:implicit} follows 
from the above display and  standard estimates on the $L^\infty$-norms of the derivatives of $h_j$, obtained from the lower bound on $|\mathrm{det}\, A|$.   
The claim that if $Z$ is a transverse complete intersection so are $V_{\vec\alpha}$, $W_{\vec\beta}$, is actually a byproduct of the application of the implicit function theorem above. 
\end{proof}
%%%%%%%%%%%%%%%%%%%%%%%%%%%%%% PROOF PROOF PROOF
Using Lemma \ref{lemma:approx} in combination with the smooth component Lemma \ref{lemma:smoothcomp}    we obtain the following approximation result by transverse complete intersections. 
%%%%%%%%%%%%%%%%%%%%%%%%%%%%%% LEMMA LEMMA LEMMA
\begin{proposition} \label{p:smoothapprox}Let $Z=\mathbf{Z}(I)$ be an algebraic variety in $\R^n$ of dimension $0\leq m\leq n$,  degree $D$, and $\{P_1,\ldots, P_J\}$ be a generating set for $I$. Let $\eps>0, R\geq 1$.  Then there exists a set $Z_0\subseteq Z$ with $\#Z_0\leq\Theta$, and collections of transverse complete intersections $\mathcal Z_\mu(Z)\coloneqq\{Z_{\mu,\vartheta}:\, 1\leq \vartheta \leq \Theta\}\subset \mathcal Z_{\mu,n} ^\times (\Theta)$, such that 
\[
\dist\left(Z \cap \mathcal A_{n}(R)\, , Z_0 \cup \bigcup_{\mu=1}^m \bigcup_{\vartheta=1}^\Theta Z_{\mu,\vartheta} \right) <\eps.
\]
The constant  $\Theta=\Theta_{m,n,J,D}$ depends only on the indicated parameters and can be explicitly computed.
\end{proposition}
%%%%%%%%%%%%%%%%%%%%%%%%%%%%%% LEMMA LEMMA LEMMA

%%%%%%%%%%%%%%%%%%%%%%%%%%%%%% PROOF PROOF PROOF
\begin{proof} We prove the proposition by induction on the dimension $m$. If $Z$ has dimension zero, then $Z\cap \mathcal A_n(R)$ is a finite set with $\#Z\leq \Theta $ points, whence the claim holds with $Z_0=Z$ and empty collections $\mathcal Z_\mu$. If $Z$ has dimension $n$, there is nothing to prove.

We now deal with the inductive step  for which some preliminary notation is needed. Let $1\leq m \leq n-1$, $Z$ and $\{P_1,\ldots, P_J\}$ be as in the statement of the lemma. Let $c=n-m$ and  $\sigma$ be a choice of increasing $c$-tuple of indices $\{j_1,\ldots, j_c\}\subset \{1,\ldots, J\}$; there are ${J}\choose{c}$ choices. Let $\vec P_\sigma:\,\R^n \to \R^c$ be the vector-valued function with $(\vec P_\sigma)_k\coloneqq   P_{j_k}$ and $\Delta_\tau \vec P_\sigma= \mathrm{det} \mathrm{D}_\tau \vec P_\sigma $, as in the statement of Lemma~\ref{lemma:approx}. Calling
\[
Z'\coloneqq Z \cap \mathbf{Z}\big(\{\Delta_\tau \vec P_\sigma:\sigma, \tau\}\big),
\]
that is $Z'$ is the complement in $Z$ of the set of smooth points $Z^{\circ m}$, we gather from Lemma \ref{lemma:smoothcomp} that $Z'\in \mathcal Z_{\bar{\mu},n}(\Theta, \Theta)$ for some $0\leq\bar \mu< m$. We now partition
\[
Z\cap \mathcal A_n(R) = U \cup W 
\]
where $W$ is the intersection of $Z\cap \mathcal A_n(R) $ with the $\eps/2$ neighborhood of $Z'$. By the inductive assumption applied to $Z'\cap A_n(R)$ with $\eps/2$ in place of $\eps$, we may find $Z_0$ with $\#Z_0\leq \Theta$ and families $\mathcal Z_{\mu}(Z')\subset \mathcal Z_{\mu,n}(\Theta)$ for $1\leq \mu \leq \bar \mu$ each with at most $\Theta$ elements so that
\begin{equation}
\label{e:appind1}
\dist\left(W \cap \mathcal A_{n}(R)\,, Z_0 \cup \bigcup_{\mu=1}^m \bigcup_{\vartheta=1}^\Theta Z_{\mu,\vartheta} \right) <\eps.
\end{equation}
We are left with treating the $U$-component. Notice that, as we have excised an $\eps/2$-neighborhood of $\mathbf{Z}(\{\Delta_\tau \vec P_\sigma:\sigma, \tau\})$, there exists $\rho>0$ such that
\[
U\subset\bigcup_\sigma U^\sigma, \qquad U^\sigma \coloneqq\Big\{x\in \mathbf{Z}(\vec P_\sigma):\, |x|\leq 2R, \, \max_{\tau} \min_{x\in U} |\Delta_\tau \vec P_\sigma(x)| \geq \rho  \Big\}.
\]
Now we may apply Lemma \ref{lemma:approx} to $\vec P=\vec P_\sigma$, with $U^\sigma$ in place of $U_{\rho,2R}$ and obtain that 
\[
\dist\left(U_\sigma,  V_{\sigma, \vec \alpha} \right) <\eps, \qquad V_{\sigma, \vec \alpha}\coloneqq \mathbf{Z}(P_{\sigma(1)}+\alpha_1,\ldots, P_{\sigma(c)}+\alpha_{c}),
\]
provided $|\vec \alpha|< \delta(\eps, \rho)$ is sufficiently small. We also note that a recursive application of \cite[Lemma 5.1]{Guth2} tells us that  the variety
$
V_{\sigma,\vec \alpha}  
$ 
is a transverse complete intersection in $\mathcal Z^\times _{m,n}(D)$ for almost all $\vec \alpha\in \mathbb R^{c}$. Choosing one such $\vec\alpha$ for each $\sigma$, we obtain
\begin{equation}
\label{e:appind2}
\dist\left(U  ,   \bigcup_{\sigma}V_{\sigma,\vec \alpha}  \right) <\eps.
\end{equation}
Combining \eqref{e:appind1} with $\eqref{e:appind2}$ completes the inductive step.
\end{proof}
%%%%%%%%%%%%%%%%%%%%%%%%%%%%%% PROOF PROOF PROOF

The next lemma will be used in the course of the proof of Theorem \ref{t:main:manifold} to control the contribution to our maximal averaging operator of (a piece of) a $m$-dimension\-al transverse complete intersection lying $s$-close to the $n-1$ dimensional hyperplane $x_n=0$, by a $m$-dimensional algebraic variety \emph{lying on} the hyperplane $x_n=0$. 

%%%%%%%%%%%%%%%%%%%%%%%%%%%%%% LEMMA LEMMA LEMMA
\begin{lemma}[approximate projection of a transverse complete intersection] \label{lemma:algproj}Let \[Z=\mathbf{Z}(P_1,\ldots, P_{n-m})\] be a transverse complete intersection in $\R^n$ with $\mathsf{deg}\, Z=D$. Let
\[
U \coloneqq Z \cap  \big\{x\in \R^n:\,  1\leq |x| <2,\, |x_n| < s\}
\]
for some $0<s< \frac12$.  There exists an algebraic variety $W$ with the following properties:
\begin{itemize}
\item[(i)] $W\subset e_n^\perp\coloneqq\{x\in \R^n:\,  x_n=0\} $;
\item[(ii)] $\mathsf{dim} W\leq m $;
\item[(iii)]$\mathsf{ct}\, W + \mathsf{deg} W\leq \Theta_{n,m,D}$;
\item[(iv)]  $\mathrm{dist}(U,W)<2s$.
\end{itemize}
\end{lemma} 
%%%%%%%%%%%%%%%%%%%%%%%%%%%%%% LEMMA LEMMA LEMMA

%%%%%%%%%%%%%%%%%%%%%%%%%%%%%% PROOF PROOF PROOF
\begin{proof} We identify $e_n^\perp$ with $\R^{n-1}$.  The proof includes a technical reduction which  we first include in our assumptions, postponing its justification until the end of the argument.
	
\vskip1mm
\paragraph{\textit{Claim.}} The ideal $I\subset\R[x_1,\ldots,x_n]$ generated by $P_1,\ldots, P_{n-m}\,$ contains a polynomial $f$  of degree $d$ where the monomial  $x_n^d$ appears with nonzero coefficient, for some $d\geq 1$. 
\vskip1mm
We proceed assuming the claim.  Define  
\[
\widehat Z\coloneqq \{z \in \mathbb C^n:\, P(z)=0 \quad \forall P \in I \}.
\]
Observe that $\widehat Z$ is a complex variety of dimension $m$ which is invariant under complex conjugation and  $Z=\widehat Z \cap \R^n$. Let $\Pi:\,\mathbb C^n \to \mathbb C^{n-1}$  indicate the canonical projection erasing the last coordinate.
By virtue of our claim we may apply \cite[Theorem 1.68]{DePf} (see also \cite[Theorem 3.3]{MP}) and obtain that 
\[
\widehat{W}\coloneqq \Pi \widehat Z = \{z \in \mathbb C^{n-1}:\,P(z)=0 \quad \forall P \in I' \},
\] 
where $I'$ is the first elimination ideal of $I$, namely $I'=I \cap \R[x_1,\ldots,x_{n-1}]$. Such a $\widehat{W}$ is a complex algebraic variety of dimension $m$. We may find a generating set of   for  $I'$ by arguing as follows. According to the results of Dub\'e \cite{Dube} and Latyshev \cite{Latyshev}, we may pick a Gr\"obner basis for $I$, which we denote $\{Q_j:1\leq j \leq J \}$, $Q_j\in\R[x_1,\ldots,x_n]$,  whose degree $\Delta\coloneqq \max \mathsf{deg}\, Q_j$, and cardinality $J$, is bounded by a constant depending  only on $D,n,m$. Using the elimination theorem \cite[Chap.\ 2 Theorem 3]{CLO}, a (Gr\"obner) basis for $I'$ is given by those $Q_j$ which do not depend on the coordinate $x_n$: this completes the claim about the generating set of $I'$. For definitions and constructive algorithms leading to Gr\"obner bases we send to the monographs \cite{CLO2,CLO} and references therein. 

At this point we denote $W\coloneqq\widehat{W} \cap \R^{n-1}$ and $W$ satisfies (i) by construction. It is immediate to verify that $W$ contains the projection of $Z$ on the last $n-1$ coordinates, whence the claim (iv) follows; in fact the stronger $\mathrm{dist}(U,W)<s$ holds. Furthermore, the fact that $W$ is a real algebraic variety of dimension $\leq m$, namely claim (ii), follows by a straightforward application of \cite[Lemma 2.1]{MP}. As $W=\mathbf{Z}(I')$, the upper bound of claim (iii) is a consequence of the above construction of a generating set for $I'$. This last assertion completes the proof of the  lemma up to the verification of the preliminary claim, which follows next.
 
We make sure that the claim holds by perturbing $Z$ slightly. For $\delta_1,\ldots, \delta_{n-1} \in \R$ consider the polynomial 
\[
f(x)\coloneqq P_1(x_1+\delta_1 x_n,\ldots, x_{n-1}+\delta_{n-1} x_{n},x_n). 
\]
As noticed in \cite[Lemma 3.4]{MP}, for all $\delta>0$ there exist $\delta_1,\ldots, \delta_{n-1}$ with $\max |\delta_j|<\delta$ such that $f$ satisfies the condition required in the claim. We take advantage of Lemma \ref{lemma:approx} and obtain that for a suitable choice of $\delta_j$ 
\[
\tilde Z\coloneqq  \mathbf{Z}(f,P_2,\ldots, P_{n-m}) 
\]
is a transverse complete intersection satisfying the claim and $\dist(U,\tilde Z)<2^{-10}s$. In particular we may find a set $\tilde U\subset \tilde Z \cap  \big\{x\in \R^n:\,  1-2^{-8}\leq |x| <2+2^{-8}, \, |x_n| < (1+2^{-8})s\}$ such that $\dist(U,\tilde U)<2^{-10}s$. Applying the above proof to $\tilde U$ in place of $U$ with the slightly different value of $s$ completes the reduction and therefore the proof of the Lemma.
\end{proof}
%%%%%%%%%%%%%%%%%%%%%%%%%%%%%% PROOF PROOF PROOF

%%%%%%%%%%%%%%%%%%%%%%%%%%%%%% SECTION SECTION SECTION
\subsection{Approximate polynomial partitioning on $\R^n$ and  proof of Proposition \ref{prop:polpart}} \label{ss:alggeom3}
We now begin the proof of Proposition \ref{prop:polpart}. One of the main tools is the following approximate polynomial partitioning theorem on $\R^n$ which can be seen as the adaptation to our setting  of the Guth-Katz polynomial partitioning theorem  \cite{GK} and of its hypersurface refinement due to Zahl \cite{Zahl}. These results are appealed to in the proof.

%%%%%%%%%%%%%%%%%%%%%%%%%%%%%% LEMMA LEMMA LEMMA
\begin{lemma}[Approximate polynomial partitioning on $\R^n$] \label{lemma:approxPP} Let $V\subset \R^n$ be a finite point set with $\# V\leq N^n$ and $\delta>0$. There exists a  polynomial $Q\in \R[x_1,\ldots,x_n]$ of degree $\leq \Theta_n E$ and   $\eps_0>0$ such that 
\begin{itemize}
\item[1.] Each of the $\leq \Theta_n E^n$ connected components of $\R^n\setminus\mathbf{Z}(Q)$ contains at most at most $(N/E)^n$ points of $V$.
\item[2.] If  $|\eps|<\eps_0$, then $\dist(V\cap \mathbf{Z}(Q),\mathbf{Z}(Q+\eps))<\delta$.
\end{itemize}
\end{lemma} 
%%%%%%%%%%%%%%%%%%%%%%%%%%%%%% LEMMA LEMMA LEMMA

%%%%%%%%%%%%%%%%%%%%%%%%%%%%%% PROOF PROOF PROOF
\begin{proof} {We use the polynomial partitioning theorem of Guth and Katz, \cite[Theorem 4.1]{GK}, with the improvement of Zahl from \cite[Corollary 2.3]{Zahl2}.  The theorem of Guth and Katz yields the existence of $Q\in \mathbb \R[x_1,\ldots,x_n] $ of degree $\leq \Theta_n E$ such that 
each of the $\Theta_n E^n$ connected components of $\R^n\setminus  \mathbf{Z}(Q)$ contains at most $(N/E)^n$ points. Now the improvement of Zahl implies that we can assume that every irreducible component $\tilde Q$ of $Q$ satisfies $\mathrm{dim}(\tilde Q)=n-1$ and $\nabla {\tilde Q}$ does not vanish identically on $\mathbf{Z}(\tilde Q)$.} In particular the latter condition implies that  $\mathbf{Z}(Q)$ is $(n-1)$-dimensional and smooth points are dense in $\mathbf{Z}(Q)$. Therefore for all $v\in V \cap \mathbf{Z}(Q)$ we may find a smooth point $w(v)\in \mathbf{Z}(Q) $ with $|v-w(v)|<2^{-1}\delta$. By Lemma
\ref{lemma:approx} and a compactness argument there exists $\eps_0>0$ such that for  all $|\eps|<\eps_0$ 
\[
\sup_{v\in V \cap \mathbf{Z}(Q) } \dist (w(v),\mathbf{Z}(Q+\eps)) <\frac{\delta}{2}.
\]
The above estimate together with the construction of $w(v)$ prove the second claim.
\end{proof}
%%%%%%%%%%%%%%%%%%%%%%%%%%%%%% PROOF PROOF PROOF

The proof proper of Proposition \ref{prop:polpart} is articulated into several steps.
\vskip2mm  
\paragraph{\textsc{Step 1} (reduction to coordinate patches). } This step is an adaptation of \cite[Subsection 5.3]{Guth2} by Guth. For an  element $g $ of the Grassmanian $\Lambda^m\R^{n}$ define the polynomial
\[
P_g (x)\coloneqq  \nabla P_{1}(x) \wedge \cdots \wedge\nabla P_{n-m}(x) \wedge g
\]
which has degree $\leq \Theta_{n,m} D$.
It is proved in \cite[Lemma 5.6]{Guth2} that one may choose a $\frac{1}{100}$-net $G\subset \Lambda^m\R^{n} $ so that, for every $g\in G$, we have that
\[
W_g \coloneqq \mathbf{Z}(P_1,\ldots, P_{n-m},P_g )
\]
is a transverse complete intersection of dimension $m-1$ and degree $\leq \Theta_{n,m} D$. Before moving further we operate a first excision from $V$, namely we set
\begin{equation}
\label{eq:vg} V_G\coloneqq \left\{v\in V:\,\inf_{g \in G} \dist(v, W_g) <\delta\right\}.
\end{equation}
Clearly, if $v\in V\setminus V_G$, then $v$ belongs to a connected component $O\in\mathcal O$ of the set  $Z\setminus \cup\{Z_g:\, g\in G\} $. Notice that there are at most $ \Theta_{n,m} D^n $ elements in $\mathcal O$. We thus partition 
\begin{equation}
\label{eq:vminusvg} V\setminus V_G= \bigcup_{O \in \mathcal O} V_O, \qquad V_O\coloneqq V\cap O.
\end{equation}

\vskip2mm
\paragraph{\textsc{Step 2} (partition of each coordinate patch).} Fix one such connected component $O \in\mathcal O$.
Proceeding exactly like in \cite[Lemma 5.6]{Guth2} we notice that $O$ is the graph of a Lipschitz map $h:O'\to \R^{n-m}$ with Lipschitz constant $<\frac{1}{10}$, where  $O'$ is an open connected subset of $ \R^m$. By rotational invariance we may identify $\R^m\equiv \mathrm{span}\{e_1,\ldots,e_m\}$ and write
\[
O\coloneqq \big\{(y_1,\ldots,y_m,h(y_1,\ldots,y_m)):\,(y_1,\ldots,y_m)\in O'\big\}.
\]
Let $\Pi:\R^{n}\to \R^m$, $\Pi(x_1,\ldots,x_n)\coloneqq (x_1,\ldots, x_m)$; polynomials of $ \mathbb R [x_1,\ldots,x_m]$ will be identified with a corresponding polynomial in $\R [x_1,\ldots,x_n]$ by precomposing with $\Pi$ without explicit mention.

   Observe that $\Pi $ is injective on $O$ and set $V'_O\coloneqq\Pi V_O \subset O'\subset \R^m$. We may apply Lemma \ref{lemma:approxPP} to the set $V'_O$, with  $m$ in place of $n$ and $\delta/2$ in place of $\delta$, in order to find a polynomial $Q_O\in \mathbb R [x_1,\ldots,x_m]$ with the following properties: $Q_O$ partitions $O'$ into $\leq \Theta_{m} E^m$ cells $\mathsf{C}'$, the connected components of  $O'\setminus \mathbf{Z}(Q_O)$,  with the property that each $\mathsf{C}'$ contains at most $(N/E)^m$ points of $V'_O$; and $V'_O\cap \mathbf{Z}(Q_O)$ is contained in the $(\delta/2)$-neighborhood of $ \mathbf{Z}(Q_O+\eps) $ in $  \mathbb{R}^m$ for all sufficiently small $\varepsilon$. By \cite[Lemma 5.1]{Guth2} we may choose $\varepsilon$ such that the above property holds and \begin{equation}
\label{e:wo}
W_O\coloneqq\mathbf{Z}(P_1,\ldots, P_m, Q_O+\eps) 
\end{equation}
is a transverse complete intersection.
 It follows that the $\leq \Theta_m E^m$ cells
\[
\mathsf{C}\coloneqq \{(y_1,\ldots,y_m,h(y_1,\ldots,y_m)):\,(y_1,\ldots,y_m)\in \mathsf{C}'\}\subset O,
\] 
which we group into $\mathsf{C}\in \vec{\mathsf{C}}_O$, are the connected components of \[O\setminus \mathbf{Z}(P_1,\ldots,P_{n-m},Q_O)\] and the sets\begin{equation}
\label{e:defVc}
V_{\mathsf{C}} \coloneqq V_O \cap \mathsf{C}
\end{equation}
have at most $  (N/E)^m$ elements. It remains to take care of those points $v\in V_O$ such that $\Pi v \in \mathbf{Z}(Q_O)$. We call this set $V_{O,\times}$. By the above, we may pick $w'=(w_1,\ldots, w_m)\in \mathbf{Z}(Q_O+\eps) $ such that $|\Pi v-w|<\delta/2$. Notice that the point $w=(w',h(w'))$ belongs to $W_O $ defined in \eqref{e:wo}  and satisfies  $| v-w|\leq |\Pi v- w'|+|h(\Pi v)-h(w')|<\delta$ using the Lipschitz constant of $h$. We have proved that
\begin{equation}
\label{e:disttci}
\sup_{v\in V_{\times,O}} \dist(v, W_O) <\delta.
\end{equation}

\vskip2mm 
\paragraph{\textsc{Step 3} (definition of $V_\times$).} By the constructions in Steps 1 and 2, the set
\[
V_\times \coloneqq  V_G \cup \bigcup_{O \in \mathcal O}V_{\times,O}
\]
satisfies the required claim, and the collection of transverse complete intersections $W_j$ of the statement is obtained by putting together the $\leq \Theta_{m,n}$ transverse complete intersections $\{W_g:\,g \in G\}$ of Step 1 with the $\leq \Theta_{m,n} D^n$ transverse complete intersections $\{W_O:\,O \in \mathcal O\}$. 

\vskip2mm \noindent \textsc{Step 4} (definition of $V_\circ$ and counting of the hyperplane crossings). 
We define 
\[
\vec{\mathsf{C}} \coloneqq \bigcup\big\{\vec{\mathsf{C}}_O:\,O \in \mathcal O\big\}.
\]
By the construction in Step 2, the sets $V_{\mathsf C}$ defined in \eqref{e:defVc} exhaust $V\setminus V_\times$ and satisfy the cardinality requirement. We are left with counting how many of the intersections $\{\mathsf{C}\cap \{x\in \R^n:\xi\cdot x=a \}:\,\mathsf{C} \in \vec{\mathsf{C}}\} $ are nontrivial. As we have at most $\leq \Theta_m D^n$ elements $O\in \mathcal O$, it suffices to show that for almost every $\xi\in \R^n, a\in \R $ the intersection $\{\mathsf{C}\cap \{x\in \R^n:\xi\cdot x=a \}$ is nontrivial for at most $\leq \Theta_{n,m} D^n E^{m-1}$ connected components $\mathsf{C} \in \vec{\mathsf{C}}_O\}$. 

Fixing  one such $O$, we may go back to the coordinate system where $O$ is a graph over the first $m$ variables. Consider $P_{\xi,a}(x_1,\ldots,x_m)=\Pi \xi \cdot x -a$ as an element of $\R[x_1,\ldots, x_m]$; we  prove the claim for all $\xi \in\R^n $  such that
$
Z_{\xi}\coloneqq \mathbf{Z}(P_1,\ldots, P_{n-m},P_{\xi,a})$ is a transverse complete intersection of dimension $m-1$, as these $\xi $ are a set whose complement is a null set in $\R^n$.    

We first count how many $\mathsf{C} \in \vec{\mathsf{C}}_O$ satisfy the condition  
\[
\mathsf{C}\cap \mathbf{Z}(P_{\xi,a}) \neq \varnothing, \qquad \partial \mathsf{C}\cap \mathbf{Z}(P_{\xi,a}) =\varnothing.
\]
This number is controlled by the number of connected components of the set $\mathbf{Z}(P_{\xi,a})\setminus Z$, which is controlled by $\Theta_n D^{n-1}$: to see this apply \cite[Theorem 4.11]{Sheffer}, which is a reformulation of the main result by Barone and Basu \cite{BarBas}.
Now we count how many $\mathsf{C} \in \vec{\mathsf{C}}_O$ satisfy instead the condition  
\[
\mathsf{C}\cap \mathbf{Z}(P_{\xi,a}) \neq \varnothing, \qquad \partial \mathsf{C}\cap \mathbf{Z}(P_{\xi,a}) \neq \varnothing.
\]
This number is controlled by the number of connected components of the set 
\[
\mathbf{Z}(P_{\xi,a},P_1,\ldots,P_m)\setminus \mathbf{Z}(Q_O)
\] 
which is controlled by $\Theta_{n,m} D^n E^{m-1}$: to see this apply again \cite{BarBas} in the form of  \cite[Theorem 4.11]{Sheffer}, with $U=\mathbf{Z}(P_{\xi,a},P_1,\ldots,P_m)$, which has dimension $m-1$, and $W=\mathbf{Z}(Q_O)$, which has degree $E$. Putting together the two counting arguments above completes the last claim of the second point of Proposition \ref{prop:polpart} and therefore the proof of the proposition.

%%%%%%%%%%%%%%%%%%%%%%%%%%%%%% SECTION SECTION SECTION
\section{Proof of Theorem \ref{t:main:manifold}} \label{sec:mainpf}

%%%%%%%%%%%%%%%%%%%%%%%%%%%%%% SECTION SECTION SECTION
\subsection{Main line of proof} Both statements in the theorem will be a consequence of the following uniform estimate: for all $1\leq m<n$, $D,J\geq 1$, and $\eta>0$, there exists a constant $\Theta=\Theta_{m,n,D,J,\eta}$ such that
\begin{equation}
\label{e:s51} K_{m,n,D,J}(N)\coloneqq
\sup_{s>0}\, \sup_{Z \in \mathcal Z_{m,n}(D,\,J)}\, \sup_{\substack{V\subset Z \cap \mathcal A_{n}(\frac32) 
\\
 \#V \leq  N^m}} \left\|   \mathrm{A}
_{V,s}\circ S_1  \right\|_{L^2(\R^{n})}^2 \leq \Theta N^{ {m-1}+\eta},
\end{equation}
where $\mathcal Z_{m,n}(D,J)$ is the class of real algebraic varieties introduced in Subsection \ref{ss:alggeom1}. Notice the slightly enlarged annulus, which is for technical reasons. This reduction was made using Proposition \ref{p:CWW}  to insert the annular cutoff.
The proof of \eqref{e:s51} will be inductive, and the constants we will induct  on  are the following; for $R\in \{1,2\},N\geq 1$
\[
\begin{split} &
K^\times _{m,n,D}(N,R) \coloneqq \sup_{s>0}\sup_{Z \in \mathcal Z^\times_{m,n}(D)} \sup_{\substack{V\subset Z \cap \mathcal A_{n}(R) \\ \#V \leq  N^m}} \left\|   \mathrm{A}
_{V,s} \circ S_1   \right\|_{L^2(\R^{n})}^2 ,
\\
 & K^\times _{m,n,D}(N)\coloneqq {K^\times_{m,n,D}(N,1).}
\end{split}
\]
The approximation  Proposition \ref{lemma:smoothcomp}  allows us to relate  $ K_{m,n,D,J}(N)$ to {$K^\times_{\mu,n,\Theta}(N)$ for $1\leq \mu \leq m$ and some suitable constant $\Theta=\Theta_{m,n,J,D}$} and ultimately induct on the latter quantity only.

%%%%%%%%%%%%%%%%%%%%%%%%%%%%%% LEMMA LEMMA LEMMA
 \begin{lemma} \label{lemma:reductSmooth} There exists a constant $\Theta= \Theta_{m,n,J,D}$ such that
\[
K_{m,n,D,J}(N) \leq   \Theta \log N  \sup_{1\leq \mu\leq m} K^\times_{\mu,n,\Theta}\left(N^{\frac{m}{\mu}}\right) .
\]
\end{lemma}
%%%%%%%%%%%%%%%%%%%%%%%%%%%%%% LEMMA LEMMA LEMMA

%%%%%%%%%%%%%%%%%%%%%%%%%%%%%% PROOF PROOF PROOF
\begin{proof}  Throughout the constant $\Theta=\Theta_{m,n,D,J} $ is understood to depend only on the indicated parameters and may vary from line to line. {For every positive integer $M$ the inequality 
\begin{equation}
\label{e:KmnDN2}
 K^\times _{m,n,D}(M,2) \lesssim    K^\times _{m,n,D}(M) 
\end{equation}
}
is obvious by rescaling. So it suffices to prove that 
{
\[
K _{m,n,D,J}(N ) \leq \Theta  \log N \sup_{1\leq \mu\leq m} K^\times _{m,n,\Theta}(N^{\frac{m}{\mu}},2) .
\]
}
 To do so, by virtue of Corollary \ref{cor:largescale}, it is enough to estimate operator norms on the left hand sides of \eqref{e:s51} for $0<s<2^{-5}$. Fix such an $s$ and $Z\in  \mathcal Z_{m,n}(D,J)  $. Notice that by a straightforward application of Proposition \ref{p:smoothapprox},
we may find that there exist {$\Theta=\Theta_{m,n,D,J}$ and a set $Z_0\subset Z$} with $\#Z_0\leq\Theta$ and collections of transverse complete intersections $\mathcal Z_\mu(Z)\coloneqq\{Z_{\mu,\vartheta}:\,1\leq \vartheta \leq \Theta\}\subset \mathcal Z_{\mu,n} ^\times (\Theta) $ such that 
\[
\dist\left(Z \cap \mathcal A_{n}(3/2), Z_0 \cup \bigcup_{\mu=1}^m \bigcup_{\vartheta=1}^\Theta Z_{\mu,\vartheta} \right) <2^{-10}s.
\]
Applying Lemma  \ref{lem:closeness}   we have
\[
\begin{split}
\sup_{\substack{V\subset Z \cap \mathcal A_{n}(3/2) \\ \#V \leq  N^m}} \left\|   \mathrm{A}
_{V,s}\circ S_1   \right\|_{L^2(\R^{n})}^2   & \leq  \Theta  \sup_{1\leq \mu\leq m}\, \sup_{1\leq \vartheta\leq \Theta }\,  \sup_{\substack{V\subset Z_{\mu,\vartheta}  \cap \mathcal A_{n}(2) \\ \#V \leq  N^m}} \left\|   \mathrm{A}
_{V }   \right\|_{L^2(\R^{n})}^2 
\\
&\leq \Theta  \log N \left( \sup_{1\leq \mu\leq m} K^\times_{\mu,n,\Theta}\left(N^{\frac{m}{\mu}},2\right) \right) 
\end{split}
\]
where in the last step we applied Proposition~\ref{p:CWW}. The claim follows by taking supremum over $s$ and $Z\in  \mathcal Z_{m,n}(D,J)  $.
\end{proof} 
%%%%%%%%%%%%%%%%%%%%%%%%%%%%%% PROOF PROOF PROOF
We now show how Lemma  \ref{lemma:reductSmooth} reduces \eqref{e:s51} to the following statement; here $1\leq m<n$.
\begin{mnest}  For all $  D\geq 1$ and for all $\eta>0$ there exists a constant $\Upsilon_{m,n,D,\eta}$ such that
\begin{equation} \tag*{(TCI)$_{(m,n)}$}   
	K^\times _{m,n,D}(N) \leq \Upsilon_{m,n,D,\eta} N^{m-1+\eta} .
\end{equation}
\end{mnest}
\begin{proof}[Proof that the \emph{(TCI)}$_{(\mu,n)}$ estimate for all $1\leq \mu\leq m$ implies \eqref{e:s51}] 
Fix $D,J\geq 1$. 
Applying  Lemma \ref{lemma:reductSmooth}, we may find a constant $\Theta=\Theta_{m,n,D,J}$ such that
\[
K_{m,n,D,J}(N) \leq   \Theta  \log N \left( \sup_{1\leq \mu\leq m} K^\times _{\mu,n,\Theta}\left(N^{\frac{m}{\mu}}\right)\right)\leq \frac{\Theta N^{\frac\eta 2} }{\eta}\left( \sup_{1\leq \mu\leq m} K^\times _{\mu,n,\Theta}\left(N^{\frac{m}{\mu}}\right)\right).
\]
Using the   estimate (TCI)$_{(m,n)}$ for $\frac \eta 2>0$, we have that
\[
K^\times  _{m,n,\Theta}\left(N \right)\leq  \Upsilon_{m,n,\Theta,\frac{\eta}2} N^{m-1+\frac{\eta}{2}}
\]
while using the  {(TCI)}$_{(\mu,n)}$ estimate   with $(m-\mu)/m$ in place of $\eta$ when $\mu<m$, we have that
\[
K^\times _{\mu,n,\Theta}\left(N^{\frac{m}{\mu}} \right)\leq \Upsilon_{\mu,n,\Theta,(m-\mu)/m} N^{m-1}.
\]
Putting together the last three displays completes the proof of the implication.
\end{proof}

%%%%%%%%%%%%%%%%%%%%%%%%%%%%%% PROOF PROOF PROOF
 
%%%%%%%%%%%%%%%%%%%%%%%%%%%%%% PROOF PROOF PROOF

Summarizing, we have reduced Theorem \ref{t:main:manifold} to proving (TCI)$_{(m,n)}$ for all pairs $(m,n)$ with $1\leq m<n$. This will be done by  the induction scheme 
\[
(m-1,n) \wedge \bigwedge_{\mu=1}^m (\mu,n-1) \implies   (m,n),\quad 1<m<n-1,
\] 
while using the two base cases $(1,n)$ and $(n-1,n)$ as seeds of our induction. The main inductive step is summarized in the following   estimate which is a consequence of polynomial partitioning.
\begin{mnpest} 
Let  $1< m<n-1$ and $ D\geq 1$.
For all integers   $E$ with $E^m\geq D^n$, there holds  \begin{equation}
 \tag*{(PART)$_{(m,n)}$}\begin{split}
 K^\times _{m,n,D}(N) & \leq 
\Theta_{m,n} D^{2n}\bigg[E^{m-1} K^\times _{m,n,D}\left({\textstyle\frac{N}{E}}\right) + (\log N)K^\times _{m-1,n,\Theta_{m,n}E}\left( N^\frac{m}{m-1}\right) 
\\[1.5ex]
&\quad  + E(\log N)^2\sup_{1\leq \mu \leq m } K^\times _{\mu,n-1,\Theta_{m,n,D}}\left(N^{\frac{m}{\mu}}\right)  \bigg].
\end{split}
\end{equation}
\end{mnpest}
We detail the two base cases in Subsection \ref{ss5:1n} and \ref{ss5:n1n} and prove the  partitioning estimate (PART)$_{(m,n)}$ in the final Subsection  \ref{ss5:pe}. The main line of proof of Theorem \ref{t:main:manifold} will be complete once we establish the implication
\begin{center}
(PART)$_{(m,n)}$ $\wedge $  (TCI)$_{(m-1,n)}$ $\displaystyle\wedge \bigwedge_{\mu=1}^m$ (TCI)$_{(\mu,n-1)}$  $\quad\implies\quad$
(TCI)$_{(m,n)}$ 
\end{center}
for all $1<m<n-1$.

%%%%%%%%%%%%%%%%%%%%%%%%%%%%%% PROOF PROOF PROOF
\begin{proof}[Proof of the implication]  Fix $D\geq 1,\eta>0$. The goal is to find a uniform in $N$ bound for
\[
\tau(N)\coloneqq \frac{K^\times _{m,n, D}(N) }{N^{m-1+\eta}}.
\]
For each   $E\geq 1$ we may use the induction assumptions  (TCI)$_{(m-1,n)}$, (TCI)$_{(\mu,n-1)}$ for $1\leq \mu \leq m$ to get the estimates
\[
\begin{split}
&
\frac{K^\times _{m-1,n,\Theta_{m,n}E}\left( N^\frac{m}{m-1}\right)}{N^{m-1}} \leq \Upsilon_{m-1,n,\Theta_{m,n}E,\frac1m},  
\\
& \frac{K^\times _{\mu,n-1, \Theta_{m,n,D}}\left(N^{\frac{m}{\mu}}\right)}{N^{m-1}} \leq   \Upsilon_{\mu,n-1,\Theta_{m,n,D},\frac{m-\mu}{m}}  , \qquad 1\leq \mu \leq m-1,
\\ 
& \frac{K^\times _{m,n-1, \Theta_{m,n,D}}\left(N\right)}{N^{m-1}} \leq   \Upsilon_{m,n-1,\Theta_{m,n,D},\frac{\eta}{2}}   N^{\frac{\eta}{2}}.
\end{split}
\]
By virtue of these estimates and of the elementary bound $\log N\leq \frac4\eta N^{\frac{\eta}{4}}$, dividing (PART)$_{(m,n)}$ by $N^{m-1+\eta}$ yields that for all $E^m>D^n$
\begin{equation}
\label{e:main:finalp1}
\begin{split}
 \tau(N)& \leq  \frac{\Theta_{m,n} D^{2n}}{E^\eta}\tau \left({\textstyle\frac{N}{E}}\right) \\ & \quad + 
\frac{\Theta_{m,n}D^{2n}}{\eta^2} \Bigg[  \Upsilon_{m-1,n,\Theta_{m,n}E,\frac1m}  \\ & \qquad \qquad \qquad\;+  \Upsilon_{m,n-1,\Theta_{m,n,D},\frac{\eta}{2}}  + \sup_{1\leq \mu\leq m-1} \Upsilon_{\mu,n-1,\Theta_{m,n,D},\frac{m-\mu}{m}}  \Bigg].
\end{split}
\end{equation}
Now choosing  \[
\begin{split} &
E=E(m,n,D,\eta)\coloneqq\left\lceil (2\Theta_{m,n} D^{2n})^{\frac1\eta} \right\rceil,
\\[1.5ex]
&\Upsilon_{m,n,D,\eta}\coloneqq  \frac{2\Theta_{m,n}D^{2n}}{\eta^2} \Bigg[  \Upsilon_{m-1,n,\Theta_{m,n}E,\frac1m}  + \Upsilon_{m,n-1,\Theta_{m,n,D},\frac{\eta}{2}} \\ &  \qquad \qquad \qquad \qquad \qquad +\sup_{1\leq \mu\leq m-1} \Upsilon_{\mu,n-1,\Theta_{m,n,D},\frac{m-\mu}{m}}  \Bigg],
\end{split}
\]
estimate \eqref{e:main:finalp1} becomes 
\begin{equation}
\label{e:main:finalp2}
\begin{split}
\tau(N) -\frac{1}{2}\tau \left({ \frac{N}{E }}\right)   \leq \frac{ \Upsilon_{m,n,D,\eta}}{2}, 
\end{split}
\end{equation}
and induction on $N$ of \eqref{e:main:finalp2} yields the claimed (TCI)$_{(m,n)}$ for this choice of $D,\eta$.
\end{proof}
%%%%%%%%%%%%%%%%%%%%%%%%%%%%%% PROOF PROOF PROOF

%%%%%%%%%%%%%%%%%%%%%%%%%%%%%% SECTION SECTION SECTION
\subsection{The (TCI)$_{(1,n)}$ estimate} \label{ss5:1n}
To prove the (TCI)$_{(1,n)}$ estimate,  we fix a transverse complete intersection $Z=\mathbf{Z}(P_1,\ldots, P_{n-1})\in \mathcal Z_{1,n}^\times(D)$. Note that in particular $Z$ is the union of at most $D^n$ smooth connected curves $\mathcal{V}_j=\{\gamma_j(t):\,t \in I_j\}\subset \R^n$ where $I_j\subset \R$ are  parameter intervals. Now for each $\xi\in \R^{n}, a\in \R $ the function  $t\mapsto \xi \cdot \gamma_j(t)-a$ may change sign at most {$O(D_j)$} times {with $\sum_j D_j \leq D$}. The latter property allows us to  appeal to an estimate of of C\'ordoba from \cite{Cor1982} which can be summarized in the following calculation.  If $\mathcal V_j\cap V= \{v_{j,1},\ldots,v_{j,n_j}\}$ where for each $j$ the directions are in consecutive order then
\[
\begin{split}
& \|A_{Z\cap V,s}f\|_{L^2(\R^n)}\leq \|A_{Z\cap \tilde V,s}f\|_{L^2(\R^n)}+ \Bigg(\sum_{j=1} ^{D^n} \sum_{\ell=1} ^{n_j-1} \|(A_{v_{j,\ell+1},s} - A_{v_{j,\ell},s})f \|_{L^2(\R^n)} ^2\Bigg)^{\frac12}
\\
&\qquad+\Bigg(\sum_{j=1} ^{D^n}  \|(A_{v_{j+1,1},s} - A_{v_{j,n_j},s})f \|_{L^2(\R^n)} ^2\Bigg)^{\frac12}
\end{split}
\]
where $\# \tilde V\leq \# V/2$. The first square function is estimated as in the proof of the good part in Theorem~\ref{thm:curve} by a constant multiple of $\|\psi\|_{\mathrm{BV}}D^{\frac 1 2}\|f\|_{L^2(\R^n)}$, using the $D_j$ crossing property of each $\mathcal V_j$; here we remember that $\psi$ is the smooth function used to define $A_{v,s}$. The square function of the second summand is estimated just by triangle inequality and is bounded by a constant multiple of $D^{\frac n 2}\|f\|_{L^2(\R^n)}$. Thus induction on $\#V$ yields the bound $\|A_{Z\cap V,s}\|_{L^2(\R^n)}\lesssim_{n}  D^{\frac n2}\log (\#V\cap Z) $. Note that the argument above is essentially identical to the one used for the good part in the proof of Theorem~\ref{thm:curve}; here however we can afford implicit losses depending on $D$. The estimate described above is summarized in the following proposition.

%%%%%%%%%%%%%%%%%%%%%%%%%%%%%% PROPOSITION PROPOSITION PROPOSITION
\begin{proposition} Let $\mathcal{V}=\{\gamma(t):\,t \in I\}\subset \R^n$ be the image of a smooth curve $
\gamma$ with the property that for each $\xi\in \R^{n}, a\in \R $, the function $t\mapsto \xi \cdot \gamma(t)-a$ changes sign at most $D$ times. Then
\[
\sup_{\substack{V\subset \mathcal V \cap\mathcal A_n(1) \\ \#V \leq N}} \left\|\mathrm{A}_{V} \right\|_{L^2(\R^{n}) } \leq \Theta_n D^{\frac{n}{2}} \log N   .
\]
\end{proposition}
%%%%%%%%%%%%%%%%%%%%%%%%%%%%%% PROPOSITION PROPOSITION PROPOSITION

In consequence of this proposition we have
\[
K^\times _{1,n,D}(N) \leq \Theta_n D^{2n} (\log N)^2 \leq \frac{\Theta_n D^{2n}}{\eta^2} N^\eta
\]
for all $\eta>0$, which complies with the claimed (TCI)$_{(1,n)}$ estimate.

%%%%%%%%%%%%%%%%%%%%%%%%%%%%%% SECTION SECTION SECTION
\subsection{The (TCI)$_{(n-1,n)}$ estimate}  \label{ss5:n1n} For the codimension $1$ case, we first observe that whenever $V\subset \mathcal A_{n}(1)$, denoting $V'=\{v':v\in V\}\subset \mathbb S^{n-1}$ we have
\[
\|M_{V }\|_{L^2(\R^n)} \leq 2  \|M_{V' }\|_{L^2(\R^n)}.
\]  
By virtue of the above display and of Lemma \ref{l:roughsmooth}, we obtain that
\begin{equation}
\label{e:5sphere}
K^\times _{n-1,n,D}(N) \leq \Theta \sup_{0<s<2^{-5}} \sup_{\substack{V\subset \SS^{n-1}\\ \#V\leq N^{n-1}}}   \left\| \mathrm{A}_{V,s} \circ S_1 \right\|_{L^2(\R^n)}^2.
\end{equation} 
We know, respectively from   \cite{KB} and from Theorem \ref{t:main:card} with $k=1$ (say),  that the supremum in the right hand side of \eqref{e:5sphere} is bounded  by $\Theta \log N\leq \frac{\Theta}{\eta}  N^{\eta} $ when $n=2$ and by $\Theta N \log N \leq \frac{\Theta}{\eta}\Theta N^{1+\eta}  $ when $n=3$. These observations deal with the cases $n=2,3$ of (TCI)$_{(n-1,n)}$. 

We now  estimate (likely not optimally) that the right hand side of \eqref{e:5sphere}  is controlled by $\Theta_n N^{{n-2}} (\log N)^{n-2}$ when $n\geq 4$, completing the proof of the (TCI)$_{(n-1,n)}$ estimate in all cases.
 Fix $0<s<2^{-5}$ and take a $2^{-10}s$-net in $\Omega$ and for each $\xi \in \Omega$ we define the cluster $V_\xi= V\cap  R_{\xi,3s}$. We say that a cluster is bad if $\#V_\xi\geq N^{n-2}$.  By a greedy selection algorithm we can identify a set  $\Omega_{\mathsf{bad}}\subset \Omega$  having at most $N$ elements, and disjoint sets $W_\xi\subset V_\xi$ with the property that, defining,
\[
 V_{\mathsf{bad}}\coloneqq
\bigcup_{\xi \in \Omega_{\mathsf{bad}}} V_{\xi} = \bigcup_{\xi \in \Omega_{\mathsf{bad}}} W_{\xi}, \qquad
 V_{\mathsf{good}}\coloneqq V\setminus \bigcup_{\xi \in \Omega_{\mathsf{bad}}} V_{\xi} ,
\]
there holds
\[
\sup_{\xi\in \Omega}\#( V_{\mathsf{good}}) \cap R_{\xi, s} \leq N^{n-2}.
\]
An easy overlap estimate on the Fourier support of the multipliers $\{A_{v,s}:v\in  V_{\mathsf{good}}\}$ then leads to the bound
\[
\left\| \mathrm{A}_{ V_{\mathsf{good}},s} \circ S_1 f \right\|_{L^2(\R^n)}^2 \leq \Theta_n  N^{n-2}\left\|f\right\|_2^2.
\]
For the directions in $V\setminus  V_{\mathsf{good}}$ we use  an  $n-2$ dimensional estimate for each $W_\xi$, as $W_\xi$ is $3s$-close to a set   $U_\xi\subset \SS^{n-1}\cap \xi^\perp\equiv {\SS^{n-2}}$ with $\#W_\xi =\# U_\xi$. Arguing inductively with the help of Proposition \ref{p:CWW}, we know \[
 \left\| \mathrm{A}_{ U_\xi }  \right\|_{L^2(\R^{n-1})}^2\leq\Theta_n (\# U_\xi)^{\frac{n-3}{n-2}} (\log N)^{n-2},\] one more log factor than the inductive estimate we assumed for $\mathrm{A}_{ U_\xi,s}\circ S_1$.   In fact, applying Corollary \ref{cor:descent},
\begin{equation}
\label{e:5sphere1}
\begin{split} & \quad
\left\| \mathrm{A}_{  V_{\mathsf{bad}},s} \circ S_1 f\right\|_{L^2(\R^n)}^2 \leq \Theta_n  \sum_{\xi \in \Omega_{\mathsf{bad}}} \left\| \mathrm{A}_{ W_\xi,s} \circ S_1 f\right\|_{L^2(\R^n)}^2 \\ & \leq \Theta_n   \left(\sum_{\xi \in \Omega_{\mathsf{bad}}}  \left\| \mathrm{A}_{ U_\xi }  \right\|_{L^2(\R^{n-1})}^2 \right) \left\|S_1f\right\|_2^2  \leq \Theta_n (\log N)^{n-2}\left( \sum_{\xi \in \Omega_{\mathsf{bad}}} (\# U_\xi)^{\frac{n-3}{n-2}} \right) \left\|S_1f\right\|_2^2
\\ & \leq \Theta_n N^{n-3}(\log N)^{n-2} \left( \sum_{k=1}^{\log N}2^{k\frac{n-3}{n-2}} \#\{\xi \in \Omega_{\mathsf{bad}}:\, \# U_\xi\sim 2^{k} N^{n-2} \} \right) \left\|S_1f\right\|_2^2\\ &\leq \Theta_n ( N\log N)^{n-2}  \left\|S_1f\right\|_2^2
\end{split}
\end{equation}
as $\#\{\xi \in \Omega_{\mathsf{bad}}:\,2^{k-1} N^{n-2}< \# U_\xi\leq 2^{k} N^{n-2} \} \leq 2^{-k}N$.

Note that if we apply the approach above for directions on $\SS^2$ we get the bound $N^\frac12 (\log N)^\frac12$ for the single annulus estimate,  which has been proved in \cite{Dem12} with similar ideas. This  is worse than the bound proved in Proposition~\ref{p:single:card} where  the factor of $\sqrt{\log N}$ is replaced by an arbitrary iterated logarithm.

%%%%%%%%%%%%%%%%%%%%%%%%%%%%%% SECTION SECTION SECTION
\subsection{Proof of the  (PART)$_{(m,n)}$ estimate}  \label{ss5:pe}
Let $D\geq 1$, $E^m\geq D^n$ and \[Z=\mathbf{Z}(P_1,\ldots, P_{n-m})\in \mathcal Z^\times_{m,n}(D)\] be a TCI. 
By Corollary \ref{cor:largescale} it is enough to argue for  $0<s<2^{-5}$. Fix  a subset $V\subset Z \cap \mathcal A_{n}(1)$ with $\# V\leq N^m$. To estimate the maximal operator $A_{V,s}\circ S_1$, we will partition $V$ using Proposition \ref{prop:polpart}, choosing $\delta=2^{-10}s$. We obtain the decomposition
\begin{equation}
\label{e:pf560}
V = V_{\circ} \cup V_{\times}
\end{equation}
with the following properties. To begin with, we may find $\leq \Theta_{m,n} D^n $ transverse complete intersections $W_j=\mathbf{Z}(P_{1},\ldots, P_{n-m},Q_j)$ of dimension $m-1$ and degree $\leq \Theta_{m,n}  E$ such that
\begin{equation}
\label{e:pf561}
\sup_{v \in V_{\times}} \inf_{j} \dist(v,W_j) <2^{-10}s.
\end{equation}
Moreover, there exist $\leq \Theta_{m,n} D^n  E^m$ disjoint connected subsets of  $Z$, denoted by $\mathsf{C} \in \vec {\mathsf{C}}  $,   with the property that
\begin{equation}
\label{e:pf562}
 V_\circ = \bigcup_{\mathsf{C} \in \vec {\mathsf{C}}} V_{\mathsf{C}}, \qquad V_{\mathsf{C}}\coloneqq V_{\circ}\cap \mathsf{C}, \qquad \#V_{\mathsf{C} }\leq \left(\frac N E\right)^m,
 \end{equation}
and such that for almost every $\xi \in \R^n$ and every  $a\in \R$, there holds
\begin{equation}
\label{e:5bcp}
\# \big\{\mathsf{C} \in \vec {\mathsf{C}}:\,\mathsf{C} \cap \{x\in \R^n:\xi\cdot x=a \} \neq \varnothing\big\}\leq \Theta_{m,n} D^{2n}E^{m-1}.
\end{equation}

%%%%%%%%%%%%%%%%%%%%%%%%%%%%%% SECTION SECTION SECTION
\subsubsection*{Controlling $V_\times$} In a similar way to what was done in the proof of Theorem \ref{t:main:card}, we approximate directions from $V_\times$ with directions from the $\leq \Theta_{m,n} D^n$ transverse complete intersections $   W_j$ which are lower-dimensional.
To wit, we use \eqref{e:pf561} to find subsets $V_j\subset W_j$ with $\#V_j \leq N^m$ so that for each  $v\in V_\times$ there exists $j$ and $w(v)\in V_j$ with $|v-w(v)|<2^{-10}s$. In particular $V_j \subset \mathcal{A}_n(2)$. Applying the approximation of Lemma~\ref{lem:closeness}, followed by Proposition \ref{p:CWW}, we deduce that
\begin{equation}\label{e:5wall}
\begin{split}
\big\|  \mathrm{A}_{V_\times,s} \circ {S_1} f  \big\|_2^2 &  \leq  \Theta_{m,n} D^n \left( \sup_{j} \big\|  \mathrm{A}_{V_j} \big\|_{L^2(\R^n)}^2 \right) \left\|{S_1} f  \right\|_2^2 
\\ 
& \leq    \Theta_{m,n} D^n K^\times _{m-1,n,\Theta_{m,n}E}(N^{\frac {m}{m-1}},2) \log N  \left\|{S_1}f\right\|_2^2
\\ 
& \leq    \Theta_{m,n} D^n K^\times _{m-1,n,\Theta_{m,n}E}(N^{\frac {m}{m-1}}) \log N  \left\|{S_1}f\right\|_2^2 \end{split}
\end{equation}
as  each $W_j\in\mathcal Z_{m-1,n}^\times(\Theta_{m,n} E)$, and using  inequality \eqref{e:KmnDN2} in the last step. This completes the control of the  $V_\times$ component of $V$.

%%%%%%%%%%%%%%%%%%%%%%%%%%%%%% SECTION SECTION SECTION
\subsubsection*{Splitting $V_\circ$ into good and bad clusters}
 Using that  $\vec{\mathsf{C}}$ is a finite set,  we may choose a $2^{-10}s$-net $\Omega$ on $\mathbb{S}^{n-1}$such that the bounded crossing property \eqref{e:5bcp} holds with $a=\pm 3s$ for \emph{all} $\xi \in \Omega$. For $\xi\in \Omega$  we define $\mathsf{k}_\xi\subset \vec{\mathsf{C}}$ to be a cluster of cells with top $\xi \in \Omega$ if
\[
\mathsf{C} \subset R_{\xi,3s} \qquad \forall \mathsf{C}\in {\mathsf{k}}_\xi.
\]
We say that a cluster ${\mathsf{k}}_\xi$ is bad if it contains more than $E^{m-1}$ cells. As $\# \vec{\mathsf{C}}\leq \Theta_{m,n}D^{n} E^m$, by  an iterative selection algorithm, we may construct $\Omega_{\mathsf{bad}}\subset \Omega$ with   $\#\Omega_{\mathsf{bad}}\leq \Theta_{m,n} D^nE$ such that setting
\[
\vec{\mathsf{C}}_{\mathsf{good}} \coloneqq \vec{\mathsf{C}} \setminus \bigcup_{\xi \in\Omega_{\mathsf{bad}} }\mathsf{k}_\xi,
\]
there holds
\begin{equation}
\label{e:5nottoomuchcont}
\#\left\{\mathsf{C}\in \vec{\mathsf{C}}_{\mathsf{good}}:\,\mathsf{C} \subset R_{\xi,3s} \right\} \leq  E^{m-1} \qquad 
\forall \xi \in \Omega.
\end{equation}
We split accordingly
\[
\begin{split} &
V_\xi  \coloneqq \bigcup_{\mathsf{C}\in \mathsf{k}_\xi} V_{\mathsf{C}}
,\qquad 
V_{\circ, \mathsf{bad}}  \coloneqq \bigcup_{\xi \in\Omega_{\mathsf{bad}} }V_{\xi},
\qquad
V_{\circ, \mathsf{good}}   \coloneqq\bigcup_{\mathsf{C}\in \vec{\mathsf{C}}_{\mathsf{good}}} V_{\mathsf{C}}. 
\end{split}
\]

%%%%%%%%%%%%%%%%%%%%%%%%%%%%%% SECTION SECTION SECTION
\subsubsection*{Controlling $ V_{\circ, \mathsf{bad}}$} We still aim to control the contribution of the $\leq\Theta_{m,n} D^n E$ bad clusters with a  lower dimensional estimate, replacing the $\leq N^m$ vectors $V_{\xi}$ with a set $U_\xi$ consisting of $\leq N^m$ vectors   from the $(n-1)$-dimensional hyperplane $\xi^\perp$. This step is more difficult than the analogous one from the proof of Theorem \ref{t:main:card} because we need to make sure that the obtained $U_\xi$ is contained in an $m$-dimensional algebraic variety  on $\xi^\perp$ of controlled degree and count. This is not in general possible by simply taking projections and using Lemma \ref{cor:descent} as the projection of an algebraic variety is in general a semi-algebraic set \cite{BCR}.

We tackle this difficulty  by applying the approximate projection Lemma \ref{lemma:algproj}, after a rotation taking $\xi$ to $e_n$,  with choice of $U$ given by
$
 U \coloneqq Z\cap A_{n}(1) \cap R_{\xi,3s}
$ 
Notice that the latter set contains $V_\xi$. Then Lemma~\ref{lemma:algproj} yields the existence of an algebraic variety $W_\xi\in \mathcal Z_{m,n-1}(\Theta_{m,n,D}, \Theta_{m,n,D}) $ contained in $\xi^\perp$ such that for every  $v\in V_\xi$ we may  find $u(v) \in W_\xi$ with $|v-u(v)|\leq 6s$. Define $U_\xi= \{u(v):v\in V_\xi\}\subset W_\xi\cap \mathcal A_{n-1}(3/2)$, and observe that $\#U_\xi\leq N^m$. Using  that there are at most  $\Theta_{m,n}D^n  E $ bad clusters together with the approximation Lemma~\ref{lem:closeness}  in the first step, Fubini's theorem in the second, Proposition \ref{p:CWW} in the third,   we obtain the chain of inequalities
\begin{equation}\label{e:5badbound}
\begin{split}
\big\|  A_{V_{\circ, \mathsf{bad}},s} \circ {S_1} f  \big\|_2^2 & \leq  \Theta_{m,n}D^n  E\left( \sup_{ \xi \in\Omega_{\mathsf{bad}
}} \big\|  A_{U_\xi}  \big\|_{L^2(\R^n)}^2 \right) \left\|{S_1} f\right\|_2^2
\\ 
& \leq   \Theta_{m,n}D^n E \left(\sup_{ \xi \in\Omega_{\mathsf{bad}}}  \big\|  A_{U_\xi }\|_{L^2(\R^{n-1})}^2\right)  \left\|{S_1}f\right\|_2^2
\\ 
& \leq   \Theta_{m,n}D^n E \log N \left(\sup_{ \xi \in\Omega_{\mathsf{bad}}} \sup_{\sigma>0} \big\|  A_{U_\xi,\sigma}\circ S_1\|_{L^2(\R^{n-1})}^2\right)  \left\|{S_1}f\right\|_2^2
\\
& \leq \Theta_{m,n}D^n E \log N  \cdot K_{m,n-1,\Theta_{m,n,D},\Theta_{m,n,D}}(N)  \left\|{S_1}f\right\|_2^2 
\\
&\leq \Theta_{m,n}D^n E (\log N)^2 \left( \sup_{1\leq \mu \leq m } K^\times _{\mu,n-1,\Theta_{m,n,D}}\left(N^{\frac{m}{\mu}}\right) \right)\left\|{S_1}f\right\|_2^2 ,
\end{split}
\end{equation}
where we used Lemma  \ref{lemma:reductSmooth}  in the very last step. This completes the treatment of $ V_{\circ, \mathsf{bad}}$. 

%%%%%%%%%%%%%%%%%%%%%%%%%%%%%% SECTION SECTION SECTION
\subsubsection*{Controlling $ V_{\circ, \mathsf{good}}$} This estimate is not much different to the one appearing in the corresponding proof of Theorem \ref{t:main:card}. We keep when possible the same notations.
For $\mathsf{C} \in \vec{\mathsf{C}}_{\mathsf{good}}$ we define  
$R_{\mathsf{C}}$ as in \eqref{e:rc} and again notice that
 $A_{v,s} f= A_{v,s}f_{R_{\mathsf{C}}}$  whenever $v\in \mathsf{C}$.
The next step is the estimation of the overlap of the sets $\{R_{\mathsf{C}}:\,\mathsf{C} \in \vec{\mathsf{C}}_{\mathsf{good}}\}$ at a generic  $\beta\in \mathbb R^n$.
 By homogeneity in the definition of $R_{v,s}$ it is enough to take $\beta \in \SS^{n-1}$ and approximate by  $\xi \in \Omega$ with $|\xi -\beta|<2^{-10}s$.   Then the overlap of the multipliers $\cic{1}_{R_{\mathsf{C}}}$ at $\beta$  is bounded by the same right hand side of \eqref{e:rc2}.  The removal of the bad clusters when constructing ${\mathsf{C}}_{\mathsf{good}}$ ensured the property
\[
\#\left\{\mathsf{C}\in \vec{\mathsf{C}}_{\mathsf{good}}:\,\mathsf{C} \subset R_{\xi,3s} \right\} \leq E^{m-1} 
\]
so we are left with counting 
\[
\begin{split} 
&\quad \#\left\{\mathsf{C}\in \vec{\mathsf{C}}_{\mathsf{good}}:\,\mathsf{C} \cap R_{\xi,2s} \neq \varnothing, \,\mathsf{C} \not \subset R_{\xi,3s}  \right\}
\\ 
& \leq  \#\left\{\mathsf{C}\in \vec{\mathsf{C}}_{\mathsf{good}}:\,\mathsf{C} \cap \left\{x\in \R^2:\xi\cdot x=\pm 3 s  \right\} \neq \varnothing \right\} \leq \Theta_{m,n} D^{2n}  E^{m-1},
\end{split}
\]
where \eqref{e:5bcp} has been taken into account. 
Therefore arguing exactly like in \eqref{e:cell} and using that $V_\mathsf{C}$ contains at most $(N/E)^m$ directions, see \eqref{e:pf562}, we obtain
\begin{equation} \label{e:5cell}
\begin{split} 
\big\|   \mathrm{A}_{  V_{\circ, \mathsf{good}} ,s} f\big\|^2_2 &\leq    \Theta_{n,m} D^{2n} E^{m-1} K^\times _{m,n,D}(N/E)  \left\|f\right\|_2^2.
\end{split}
\end{equation} 
The sought after (PART)$_{(m,n)}$ estimate for these values of $m,n,D,N,E$ follows by collecting  \eqref{e:5wall}, \eqref{e:5badbound} and \eqref{e:5cell}.

%%%%%%%%%%%%%%%%%%%%%%%%%%%%%% SECTION SECTION SECTION
 
%%%%%%%%%%%%%%%%%%%%%%%%%%%%%% SECTION SECTION SECTION


\end{document}